\documentclass[11pt]{amsart}
\pdfoutput=1
\usepackage{amsmath,amssymb,amsthm}
\usepackage{microtype,booktabs,array,adjustbox,tikz,geometry}
\usetikzlibrary{calc}
\usepackage[colorlinks=true,linkcolor=blue,citecolor=blue,urlcolor=blue]{hyperref}
\hypersetup{pdftitle={Toward an exotic S2 x S2 and an exotic CP2 connected sum CP2-bar},
 pdfauthor={Bernd Johannes Wuebben}}
\newtheorem{theorem}{Theorem}[section]
\newtheorem{proposition}[theorem]{Proposition}
\newtheorem{lemma}[theorem]{Lemma}
\newtheorem{corollary}[theorem]{Corollary}
\theoremstyle{definition}
\newtheorem{definition}[theorem]{Definition}
\newtheorem{question}[theorem]{Question}
\theoremstyle{remark}
\newtheorem{remark}[theorem]{Remark}
\newcommand{\Z}{\mathbb Z}
\newcommand{\Q}{\mathbb Q}
\newcommand{\Ftwo}{\mathbb F_2}
\newcommand{\CP}{\mathbb{CP}}
\newcommand{\bCP}{\overline{\mathbb{CP}}}
\newcommand{\nc}[1]{\langle\!\langle #1\rangle\!\rangle}
\newcommand{\BK}{\mathcal B}
\newcommand{\partheading}[2]{\bigskip\phantomsection
\addtocontents{toc}{\protect\addvspace{6pt}}
\addcontentsline{toc}{part}{Part #1.\ #2}
\begin{center}\large\bfseries Part #1.\ #2\end{center}\medskip}
\title[Toward an exotic $S^2\times S^2$]{Toward an exotic $S^2\times S^2$
 and an exotic $\CP^2\#\bCP^2$}
\author{Bernd Johannes Wuebben}
\address{New York, NY}
\email{wuebben@gmail.com}
\date{October 16, 2026}

\begin{document}
\raggedbottom
\begin{abstract}
We study a specified Luttinger filling in the construction of Lidman
and Piccirillo whose simple connectivity would yield exotic
$S^2\times S^2$ and $\CP^2\#\bCP^2$. An earlier version of this
manuscript claimed simple connectivity from incorrectly derived group
relations; that claim is withdrawn. We derive the complete
fundamental-group presentation and express the group as a one-relator
quotient of an iterated HNN extension containing the genus-two surface
group. The filled group is trivial exactly when the final relation
kills a specified fiber commutator. We also determine its universal
central extension with cyclic kernel, construct geometric realizations
of it, and extend the boundary involution over the filling. The final
filling problem remains unresolved.
\end{abstract}
\maketitle
\clearpage
\setcounter{tocdepth}{2}
\tableofcontents
\clearpage

\section{Introduction}

The purpose of this paper is to give a starting point for further work
on a specific four-manifold construction whose fundamental group remains
unknown. We record its marked geometry, the complete group presentation,
structural reductions, and geometric restrictions on possible proofs.
The results separate what the construction already provides from the
precise fundamental-group question that prevents its completion.
The starting point is Lidman and Piccirillo's construction
\cite{LP25} of a spin rational homology four-sphere distinguished from
a comparison manifold by smooth sliceness of the figure-eight knot.
Their construction also gives closed manifolds with the integral
cohomology rings of $S^2\times S^2$ and $\CP^2\#\bCP^2$ and
nonstandard smooth behavior. Identifying their homeomorphism types
requires additional information about fundamental groups.

We fix one member of the surgery construction, denoted
\[
                  V=W(+1,-1;0,0).
\]
The parameters $(p,q;n,m)$ specify the slopes
$M(\ell C^n)^p$ and $N(DE^m)^q$ in the peripheral bases defined below;
the last two entries thus specify sections of the surgery tori.
In the convention of
Section~\ref{sec:construction}, the positive cotangent disk meridians are
$M$ and $N^{-1}$, the section longitudes are $\ell$ and $D$, and the
killed slopes are $M\ell$ and $N^{-1}D$, equivalently $M\ell$ and
$ND^{-1}$. These data complete the marking of the construction;
they are not inferred from an unspecified numerical surgery sign.
Lidman and Piccirillo explicitly allow choices of parametrization of
their tori \cite[Section 1, footnote 2]{LP25}.

Write $G=\pi_1(V)$ and $Y=\partial V$. The boundary is identified with
zero surgery on the square knot. Let $\sigma$ be the free
orientation-reversing boundary involution specified in
Section~\ref{sec:boundary}, and put
\[
             Z=V\cup_\sigma V,\qquad W=V/{\sim_\sigma}.
\]
In the definition of $W$, only boundary points are identified. Thus
$Z\to W$ is a connected double cover. The notation $W$ without
parameters denotes this closed manifold, whereas $W(p,q;n,m)$ denotes
a compact surgery piece.

\begin{theorem}[Geometric consequences]\label{thm:geometric}
For the specified filling and boundary marking the following hold.
\begin{enumerate}
\item $V$ is a spin integral homology $S^2$. The double $Z$ is symplectic
and has the integral cohomology ring of $S^2\times S^2$. A surviving
fiber $F$ and a doubled section $H$ are genus-two symplectic surfaces
forming a basis of $H_2(Z;\Z)$, with
\[
                      F^2=H^2=0,\qquad F\cdot H=1.
\]
Every nonzero square-zero class is $k[F]$ or $k[H]$ and has minimum
genus $|k|+1$. In particular $Z$ is not diffeomorphic to $S^2\times S^2$.
\item $W$ is a spin rational homology four-sphere with
$H_1(W;\Z)=\Z/2$. The figure-eight knot is not smoothly slice in $W$.
\item There is an orientation-preserving diffeomorphism $f:Y\to Y$
such that $Z_f=V\cup_{f\circ\sigma}V$ has the integral cohomology ring
of $\CP^2\#\bCP^2$ and is not diffeomorphic to that manifold.
Here $f$ is the boundary map of \cite[Figure 3]{LP25}, transported
through the specified marking of $Y$.
\item $\sigma$ extends to an orientation-reversing diffeomorphism of
$V$. Consequently either copy of $V$ is a retract of $Z$. If
$H_\partial=\operatorname{im}(\pi_1(Y)\to G)$, then
\[
              \pi_1(Z)\cong G*_{H_\partial}G,
       \qquad \pi_1(Z)=1\ \Longleftrightarrow\ G=1.
\]
The amalgamation uses the actual boundary image, not an assertion that
$\pi_1(Y)$ injects into $G$.
\end{enumerate}
\end{theorem}

Let $\BK$ denote the Kawauchi manifold used in \cite{LP25}: it is
obtained from the zero trace of the figure-eight knot by identifying
boundary points under the free involution induced by its strongly
negative-amphicheiral symmetry. We use the established facts that
$\BK$ is spin, is a rational homology four-sphere, has
$\pi_1(\BK)=\Z/2$, and contains a slice disk for the figure-eight knot.

\begin{corollary}[Conditional homeomorphism statements]\label{cor:conditional}
If $G=1$, then $Z$ is homeomorphic but not diffeomorphic to
$S^2\times S^2$, and $Z_f$ is homeomorphic but not diffeomorphic to
$\CP^2\#\bCP^2$. Moreover
\[
        W\text{ is homeomorphic to }\BK
          \quad\Longleftrightarrow\quad G=1.
\]
Whenever this condition holds, $W$ and $\BK$ are distinguished smoothly
by sliceness of the figure-eight knot.
\end{corollary}

The geometric conclusions in the first three parts of
Theorem~\ref{thm:geometric} follow the construction and arguments of
Lidman--Piccirillo. We supply the slope, boundary-marking, and normal-framing
details needed for the fixed $V$. The extension in the fourth part makes
the missing group condition necessary as well as sufficient for the
stated homeomorphism comparison with $\BK$.

Part II gives the complete presentation of $G$ and its reduction to three
generators and five relators. It also identifies an intermediate group
in which the fiber group does inject. Here is the resulting algebraic
form of the remaining question; all the words in the statement are part
of its definition.

\begin{theorem}[The last relation]\label{thm:last}
Let
\[
 \Pi=\langle x,y,r,s\mid[x,y][r,s]=1\rangle,\qquad
 d_0=y^{-1}x^{-1}s,\quad e_0=x^{-1}s,
\]
and define
\[
 Q_2=\langle\Pi,B\mid
 BrB^{-1}=r,\ BxB^{-1}=y^{-1},\ Bd_0B^{-1}=e_0\rangle.
\]
There is an automorphism $\theta$ of $Q_2$ fixing $x,y,r,B$ and
sending $s$ to $sB^{-1}$. Let $\phi$ be the automorphism of $\Pi$
interchanging $x,r$ and $y,s$, and set
\[
 K=\langle Q_2,A\mid A\theta(g)A^{-1}=\phi(g)
                      \ (g=x,y,r,s)\rangle.
\]
Both displayed extensions are injective HNN extensions. In particular
$\Pi$ embeds in $K$. With
\[
 m=yB^{-1}x^{-1}y^{-1}B,\qquad \lambda=Axm,
\]
there is an isomorphism
\[
                        G\cong K/\nc{\lambda}_K
\]
preserving the original fiber generators. The element $J=[x,y]$ has
infinite order in $K$ and normally generates $G$ after passage to this
quotient. Therefore
\[
       G=1\quad\Longleftrightarrow\quad J\in\nc{\lambda}_K.
\]
Neither membership nor nonmembership in this normal closure is proved.
\end{theorem}
The commutator $J$ represents the separating fiber curve dividing its
two handles.

The further results serve two purposes. Proposition~\ref{prop:transport}
computes the entire kernel of restoring the meridian transport before
the last filling. Theorems~\ref{thm:central} and \ref{thm:circle}, with
Propositions~\ref{prop:exterior} and \ref{prop:two-handles}, identify a
single cyclic central kernel in algebraic, bundle, and exact symplectic
terms. The accompanying limitations specify what these descriptions do
not yet control. Appendix~\ref{app:trefoil} gives another exact
presentation as a quotient of a trefoil amalgam; its two remaining
relations are displayed in full.

The lower-case word $m$ in this statement maps to the original meridian
$M$ only after the final quotient. The attaching word $\lambda=Axm$
is distinct from the section longitude $\ell=Ax$.

\subsection*{The geometric presentation and the final filling}
Based transport through a punctured square
requires a conjugated meridian correction; the corrections cannot be
removed by retaining the unpunctured monodromy relations. We use a
corner-excised square and compute all end inclusions in the drilled
bundle. The resulting fifteen relations,
including both meridian transports, are printed in
Section~\ref{sec:presentation}.

The remaining question is whether the final filling kills the normally
generating fiber commutator. Theorem~\ref{thm:last} establishes the
surface-group injection before that filling; its proof does not extend
this injection to the final quotient. Perfectness alone does not decide
whether the quotient is trivial.

\subsection*{Conventions}
Paths are composed from left to right, diffeomorphisms from right to
left, and $[g,h]=ghg^{-1}h^{-1}$. Normal closures are taken in the
group indicated by the subscript. Homology has integral coefficients
unless otherwise specified. An exotic copy of a manifold means a
smooth manifold homeomorphic but not diffeomorphic to it; this paper
makes such assertions only under the stated hypothesis $G=1$.

\subsection*{How to read the paper}
The geometric reduction is the organizing principle:
simple connectivity of the specified piece would give all three
homeomorphism conclusions. Part I proves the geometric inputs without
assuming that condition. The extension of the boundary involution makes
simple connectivity of the piece necessary as well as sufficient for
the even double and the slicing comparison. It also yields restrictions
on proper surfaces and a contractible singular orbit space.

Part II is independent of a conjectured answer for the group. It derives
every relation, proves the intermediate surface injections, and identifies
the last quotient. Readers interested primarily in the group can begin
with Theorem~\ref{thm:last}, the full presentation in
Section~\ref{sec:presentation}, and Section~\ref{sec:open}.
Part III gives two geometric realizations of the universal central
extension: a circle bundle and a quotient of the section-exterior group.
Its central fiber offers a sufficient test for nontriviality, not a
second equivalent test for triviality. The final discussion distinguishes
these logical directions and explains the restrictions on several
natural approaches.

\begin{center}
\begin{adjustbox}{max width=\textwidth}
\begin{tabular}{ll}
\toprule
Notation & Meaning \\
\midrule
$V,\ Y=\partial V$ & Fixed Luttinger filling and its outer boundary \\
$Z,\ Z_f,\ W$ & Even double, odd regluing, and boundary quotient \\
$F,\ H$ & Genus-two fiber and doubled section in $Z$ \\
$\Pi,\ G$ & Abstract surface group and $\pi_1(V)$ \\
$Q_2,\ K$ & Injective HNN extensions before the last quotient \\
$Q_c,\ H_\beta$ & Meridian vertex group and beta-only filling group \\
$\mathcal P,\ \mathcal Z$ & Circle bundle and auxiliary Liouville domain \\
$\ell=Ax,\ \lambda=Axm$ & Section longitude and final attaching word \\
$J=[x,y]$ & Normal generator of the complete group $G$ \\
$\widehat G,\ t$ & Universal central extension and its central generator \\
$X,\ Q_X,\ \mu$ & Section exterior, its group, and section meridian \\
$\mathcal O,\ \Lambda$ & Involution orbit space and its regular-part group \\
\bottomrule
\end{tabular}
\end{adjustbox}
\end{center}
The principal group maps are
\[
\begin{gathered}
 \Pi\hookrightarrow Q_c\hookrightarrow Q_2\hookrightarrow K
       \twoheadrightarrow H_\beta\twoheadrightarrow G,\\
 Q_X\twoheadrightarrow\widehat G\twoheadrightarrow G,
 \qquad \pi_1(\mathcal P)=\pi_1(\mathcal Z)=\widehat G.
\end{gathered}
\]
Here $H_\beta$ is the group after the beta filling alone. The first
line locates the last filling; the second locates the centralization
of the section meridian. Only the arrows explicitly drawn as inclusions
assert injectivity. The later notation is defined in detail when it is used. In
particular $X$ is a manifold and $\mathcal X$ is a relator;
$W$ identifies only boundary points, whereas $\mathcal O$ identifies
all orbits of an involution on $V$.

\subsection*{Acknowledgements}
The author used LLMs during the development of this work for literature exploration, editorial and coding assistance and as an adversarial reader.

\partheading{I}{The marked construction and its geometric consequences}

\section{The marked bundle and the surgery slopes}\label{sec:construction}

\subsection{The fiber and its monodromies}
Let $F$ be the quotient of a regular octagon with directed edge word
\[
                   xyx^{-1}y^{-1}rsr^{-1}s^{-1}.
\]
The vertices become a single point $p_0$, and the center becomes $O$.
Orient $F$ so that the intersection numbers of $x,y$ and of $r,s$
are $+1$. Number the directed polygon edges $E_1,\ldots,E_8$ in this
order. The following chord endpoints specify the five-chain
$a,b,c,d,e$. A parameter on an edge is measured from its initial
vertex in the polygon boundary traversal.

\begin{center}
\begin{adjustbox}{max width=\textwidth}
\begin{tabular}{cl}
\toprule
Curve & Chord endpoints \\
\midrule
$a$ & $(E_2,13/20),(E_4,7/20)$ \\
$b$ & $(E_1,13/20),(E_3,7/20)$ \\
$c$ & $(E_2,7/20),(E_8,13/20)$ and $(E_6,7/20),(E_4,13/20)$ \\
$d$ & $(E_5,13/20),(E_7,7/20)$ \\
$e$ & $(E_6,13/20),(E_8,7/20)$ \\
\bottomrule
\end{tabular}
\end{adjustbox}
\end{center}

The half-turn $\phi$ fixes $p_0,O$ and interchanges $x,r$ and $y,s$
as based loops. Choose positive Dehn twists supported away from these
points. Their based actions are
\[
 T_a(x)=x,\quad T_a(y)=yx,\qquad
 T_b(x)=xy^{-1},\quad T_b(y)=y,
\]
and they fix $r,s$. For example, the $a$-parallel at the $y$ crossing
is based by the $y$ prefix and has word $yxy^{-1}$. The corresponding
$b$-parallel has word $xyx^{-1}$. Inserting these parallels with their
opposite transverse signs gives the displayed formulas. Hence
\begin{equation}\label{eq:monodromies}
 \phi(x,y,r,s)=(r,s,x,y),\qquad
 \psi(x,y,r,s)=(y^{-1},yx,r,s),\quad \psi=T_aT_b.
\end{equation}

\begin{figure}[htbp]
\centering
\begin{adjustbox}{max width=\textwidth}
\begin{tikzpicture}[scale=1.75]
\foreach \i in {1,...,8} {\coordinate (P\i) at ({135-45*\i}:1.3);}
\coordinate (Bx) at ($(P1)!0.65!(P2)$);
\coordinate (Cy) at ($(P2)!0.35!(P3)$);
\coordinate (Ay) at ($(P2)!0.65!(P3)$);
\coordinate (Bxi) at ($(P3)!0.35!(P4)$);
\coordinate (Ayi) at ($(P4)!0.35!(P5)$);
\coordinate (Cyi) at ($(P4)!0.65!(P5)$);
\coordinate (Dr) at ($(P5)!0.65!(P6)$);
\coordinate (Cs) at ($(P6)!0.35!(P7)$);
\coordinate (Es) at ($(P6)!0.65!(P7)$);
\coordinate (Dri) at ($(P7)!0.35!(P8)$);
\coordinate (Esi) at ($(P8)!0.35!(P1)$);
\coordinate (Csi) at ($(P8)!0.65!(P1)$);
\draw (P1)--(P2)--(P3)--(P4)--(P5)--(P6)--(P7)--(P8)--cycle;
\draw[green!40!black,thick] (Ay)--(Ayi) node[pos=.58,right] {\scriptsize $a$};
\draw[orange!80!black,thick] (Bx)--(Bxi) node[pos=.75,right] {\scriptsize $b$};
\draw[blue!70!black,very thick] (Cy)--(Csi);
\draw[blue!70!black,very thick] (Cs)--(Cyi) node[pos=.55,below] {\scriptsize $c$};
\draw[orange!80!black,thick] (Dr)--(Dri) node[pos=.75,left] {\scriptsize $d$};
\draw[green!40!black,thick] (Es)--(Esi) node[pos=.58,left] {\scriptsize $e$};
\foreach \i in {1,...,8} {\fill (P\i) circle (0.025);}
\node[above right] at ($(P1)!0.5!(P2)$) {\scriptsize $E_1:x$};
\node[right] at ($(P2)!0.5!(P3)$) {\scriptsize $E_2:y$};
\node[below right] at ($(P3)!0.5!(P4)$) {\scriptsize $E_3:x^{-1}$};
\node[below right] at ($(P4)!0.5!(P5)$) {\scriptsize $E_4:y^{-1}$};
\node[below left] at ($(P5)!0.5!(P6)$) {\scriptsize $E_5:r$};
\node[left] at ($(P6)!0.5!(P7)$) {\scriptsize $E_6:s$};
\node[above left] at ($(P7)!0.5!(P8)$) {\scriptsize $E_7:r^{-1}$};
\node[above left] at ($(P8)!0.5!(P1)$) {\scriptsize $E_8:s^{-1}$};
\fill (0,0) circle (0.025);\node[below] at (0,0) {\scriptsize $O$};
\end{tikzpicture}
\end{adjustbox}
\caption{The marked fiber. All vertices are identified to $p_0$;
paired endpoints close the five curves. The half-turn exchanges
$a,e$ and $b,d$ and preserves $c$. The endpoint table, rather than
the scale of the drawing, specifies the curves.}
\label{fig:fiber}
\end{figure}

\subsection{The punctured-torus base}
Remove equal quarter-disks of radius less than $1/8$ from the four
corners of $[0,1]^2$, with coordinates $(v,w)$. Glue $v=1$ to $v=0$
using $\phi$, and $w=1$ to $w=0$ using $\psi$: explicitly
$(z,1,w)\sim(\phi z,0,w)$ and
$(z,v,1)\sim(\psi z,v,0)$. Smooth the boundary. This gives a
genus-two bundle $R$ over a once-punctured torus. There is no unremoved
corner at which the two transition maps would have to commute.
Indeed $\phi\psi(x)=s^{-1}$ whereas $\psi\phi(x)=r$.

At $(1/4,1/4)$ lift the positive horizontal and vertical base loops
$a_0,b_0$ using $p_0$, and set $A=a_0^{-1}$, $B=b_0^{-1}$.
The alpha torus $T_\alpha$ is the suspension of $c$ over $w=1/2$;
the beta torus $T_\beta$ is the suspension of $e$ over $v=1/2$.
They are disjoint, since $c$ and $e$ are disjoint.
An invariant fiber area form and a positive base area term make them
Lagrangian. The fixed points $p_0,O$ give two disjoint proper
punctured-torus sections, denoted $\Sigma,\Sigma'$, avoiding the tori.

This is a marked version of the intrinsic bundle in
\cite[Section 1]{LP25}. To identify the surface models, take a regular
neighborhood of the ordered five-chain. It has four intersection
vertices, eight edges, genus two, and two boundary components. Its
complement consists of two disks. Match intersection disks and bands
equivariantly, preserving cyclic orders. The involution is free on
the chain neighborhood and has one fixed point in each complementary
disk. The boundary identifications extend equivariantly over these
disks by passing to their branched disk quotients. This identifies the
ordered curves and $p_0,O$, and conjugates the two chosen monodromies.
It therefore extends over the base disk and its two handles.

\begin{figure}[htbp]
\centering
\begin{adjustbox}{max width=\textwidth}
\begin{tikzpicture}[scale=.9,>=stealth]
\draw (0.35,0)--(3.65,0) arc[start angle=180,end angle=90,radius=.35]
 --(4,3.65) arc[start angle=270,end angle=180,radius=.35]
 --(.35,4) arc[start angle=0,end angle=-90,radius=.35]
 --(0,.35) arc[start angle=90,end angle=0,radius=.35];
\draw[green!40!black,very thick] (0,2)--(4,2);
\draw[blue!70!black,very thick] (2,0)--(2,4);
\node[green!40!black,above] at (3.05,2) {$T_\alpha:\ c$};
\node[blue!70!black,right] at (2,3.1) {$T_\beta:\ e$};
\fill (1,1) circle(.035);
\node[below left] at (1,1) {\scriptsize $(1/4,1/4)$};
\draw[->] (1,1)--(1.7,1) node[midway,below] {\scriptsize $a_0$};
\draw[->] (1,1)--(1,1.7) node[midway,left] {\scriptsize $b_0$};
\node[left] at (0,2.8) {$\phi$};\node[right] at (4,2.8) {$\phi$};
\node[above] at (1,4) {$\psi$};\node[below] at (3,0) {$\psi$};
\node[below] at (4,-.2) {$v$};\node[left] at (-.2,4) {$w$};
\end{tikzpicture}
\end{adjustbox}
\caption{Projection to the corner-excised base square. The horizontal
circle supports the $c$ suspension and the vertical circle the $e$
suspension. Their projections cross, but the tori are disjoint because
$c\cap e=\varnothing$. The rounded arcs become the single boundary
circle after opposite sides are identified. The arrows indicate the
positive base directions; $A=a_0^{-1}$ and $B=b_0^{-1}$.}
\label{fig:base}
\end{figure}

\subsection{Peripheral paths and section choices}\label{subsec:paths}
Write $y=y_1y_2$, where $y_1$ ends at the $c$ crossing, and write
$s=s_1s_c=s_2s_3$, where $s_1$ ends at its $c$ crossing and $s_2$
at the subsequent $e$ crossing. The half-turn exchanges $y_1,s_1$.
The two half-arcs $d_c,d'_c$ on $c$, from the $s$ crossing to the
$y$ crossing, satisfy
\begin{equation}\label{eq:arcs}
          s_1d_cy_1^{-1}=r^{-1},\qquad
          s_1d'_cy_1^{-1}=x.
\end{equation}
These equalities can be checked directly in the polygon. If $w_i$
is the boundary prefix to vertex $i$, with $w_1=1$, crossing from
$E_i$ to its paired edge $E_j$ contributes $g_i=w_iw_{j+1}^{-1}$.
The two products are $w_6^{-1}g_4w_2=r^{-1}$ and
$w_6^{-1}g_6w_2=([x,y][r,s])^{-1}x$.

For the alpha torus use $y_1$ followed by a short positive vertical
interval as a common basing path for its meridian and both longitudes.
For the beta torus use $s_2$ followed by a short positive horizontal
interval. Let $M,N$ be the positive puncture boundaries in the
$(y,b_0)$ and $(s,a_0)$ transport squares, respectively. The based
fiber longitudes, second alpha meridian, and section longitudes are
\begin{equation}\label{eq:peripheral}
 C=(rx)^{-1},\quad E=sr^{-1}s^{-1},\quad K_0=r^{-1}Mr,
 \qquad \ell=Ax,\quad D=BK_0^{-1}.
\end{equation}

Here the alpha section travels negatively in the base and closes by
$d'_c$. More explicitly, if $a_c$ is positive transport from the
$s$ crossing to the $y$ crossing in a rectangle just below the alpha
torus, then $s_1a_cy_1^{-1}=a_0$. Thus
$y_1a_c^{-1}d'_cy_1^{-1}=Ax$. Closing instead by $d_c$ gives
$\ell C=Ar^{-1}$; it changes the section by one fiber turn.
The two alpha punctures have the same sign, since the half-turn
preserves the oriented core and exchanges their tracks.

The rectangle comparing the beta section to $b_0$ crosses $c$ once.
For its positive section $b_e$ it gives
$s_2b_es_2^{-1}b_0^{-1}=K_0$. Inverting the section yields
$D=BK_0^{-1}$. The terminal path $s_3^{-1}$ gives a useful independent
comparison: the entire beta triple rebased by $s^{-1}$ is
\begin{equation}\label{eq:beta-rebase}
       n=s^{-1}Ns,\qquad s^{-1}Es=r^{-1},\qquad s^{-1}Ds=B.
\end{equation}
This terminal rectangle misses $c$ and is fixed by $\psi$.
Accordingly $ND^{-1}$ becomes $nB^{-1}$, not $NB^{-1}$.

\subsection{Cotangent charts and the fixed filling}
Near each core choose area coordinates $(\theta,t)$ with form
$dt\wedge d\theta$, with $\theta$ increasing along $C$ or $E$.
Let $\kappa>0$ multiply the base area form. The following ordered
cotangent coordinates satisfy
$\omega=dQ_1\wedge dY_1+dQ_2\wedge dY_2$:
\begin{equation}\label{eq:charts}
\begin{array}{c|cc|cc}
 &Q_1&Q_2&Y_1&Y_2\\ \hline
 \alpha&\theta-v/2&-v&-t&t/2-\kappa(w-1/2)\\
 \beta&\theta&-w&-t&\kappa(v-1/2).
\end{array}
\end{equation}
Direct substitution gives $dt\wedge d\theta+\kappa\,dv\wedge dw$.
At alpha the half-turn is $\theta\mapsto\theta+1/2$, so $Q_1$
descends. The section with $Q_1$ constant and $v$ decreasing is the
chosen half-drift section $\ell$. At beta the section has $\theta$
constant and $w$ decreasing. A push-off with constant momenta in
these charts gives the Lagrangian framing and introduces no meridian
turn in either section.

\begin{definition}\label{def:V}
The positive meridians are the positively oriented boundaries of
coordinate disks in the ordered $(Y_1,Y_2)$ planes in
\eqref{eq:charts}. Use section coorientations $+C,+E$ and perform
meridian-plus-section fillings along the chosen sections. The resulting
compact manifold is $V=W(+1,-1;0,0)$.
\end{definition}

The determinants from the puncture coordinates $(t,w)$ and $(t,v)$
to $(Y_1,Y_2)$ are $+\kappa$ and $-\kappa$. The two positive
coordinate meridians are therefore $M$ and $N^{-1}$, and the slopes are
\begin{equation}\label{eq:slopes}
                         M\ell,\qquad N^{-1}D.
\end{equation}
Since the peripheral loops commute, the second word can be inverted
to give $ND^{-1}$. In the coordinates of
Auroux--Donaldson--Katzarkov \cite[Section 2.1]{ADK03}, ordered as
(section, coorientation), these are coefficient $+1$ surgeries with
coorientations $+C,+E$. To check the sign, their coordinates are
$(x_1,x_2,y_1,y_2)=(Q_2,Q_1,Y_2,Y_1)$. The interchange of the
normal coordinates reverses the normal disk orientation. Their shear
adds $k\chi(y_1)$ to $x_1$ on the top edge of the normal square;
a positively oriented boundary traverses that edge with $y_1$
decreasing and acquires $-k$ section turns. With the opposite
$(Y_1,Y_2)$ disk orientation, this gives meridian plus $k$ section
turns. Here the gluing map runs from the inserted annulus to the
exterior, so the killed slope is the image of its meridian; reversing
the direction of this map reverses the coefficient convention.
Taking $k=1$ gives \eqref{eq:slopes}. Thus these are
Luttinger surgeries and preserve the symplectic structure outside
their neighborhoods.

The word ``positive'' in Definition~\ref{def:V} matters. The normal
orientation induced by the ordered torus coordinates $(Q_1,Q_2)$ is
opposite to the coordinate orientation $(Y_1,Y_2)$, since
\[
 \omega^2=-2\,dQ_1\wedge dQ_2\wedge dY_1\wedge dY_2.
\]
Using the induced-normal meridians with the same numerical
meridian-plus-section instruction would instead give $M^{-1}\ell$
and $ND$. This is a different specification of slopes.

\begin{remark}[Changes of marking]\label{rem:conventions}
A simultaneous conjugation of an entire peripheral triple conjugates
its slope and preserves its normal closure. Inverting the whole slope
also preserves the filling instruction. In contrast, if a new section
coordinate is $\ell'=\ell C^n$, the old slope $M\ell^p$ becomes
$M(\ell')^pC^{-np}$; dropping the last factor changes the section
used for surgery. The general marked family has slopes
$M(\ell C^n)^p$ and $N(DE^m)^q$. No equivalence between distinct
members of that family is assumed here. All conclusions below concern
the particular slopes \eqref{eq:slopes}.
\end{remark}

\section{The boundary involution and the section framings}\label{sec:boundary}

Choose the based boundary loop $d$ of the punctured torus by
$[a_0,b_0]d=1$. For a base loop $u$ the fiber transport acts as
$g\mapsto u^{-1}gu$, so transport reverses the order of path
concatenation. The boundary monodromy is
\begin{equation}\label{eq:rho}
                  \rho=\phi\psi^{-1}\phi\psi.
\end{equation}
The half-turn conjugates $T_a$ to $T_e$ and $T_b$ to $T_d$.
Consequently
\[
       \psi\rho\psi^{-1}=T_aT_bT_d^{-1}T_e^{-1}.
\]
This is the zero-surgery monodromy of the square knot in
\cite[Section 1]{LP25}. Conjugating once more by $T_e^{-1}$ gives
\begin{equation}\label{eq:boundary-marking}
 \chi=T_e^{-1}T_d^{-1},\qquad
 U=T_e^{-1}\psi,\qquad U\rho U^{-1}=\delta=\psi\chi.
\end{equation}
All these twists are the identity near $p_0,O$.

\subsection{Compatible representatives}
Use the reflection $\epsilon$ of the symmetric five-chain model in
\cite[Figure 2]{LP25}. It reverses orientation, fixes each of $p_0,O$,
commutes with $\phi$, and exchanges $a,e$ and $b,d$.
Put $j=\phi\epsilon$. One concrete realization is the double of a
pair of pants: $b,d$ are the inner boundary circles, $a,e$ are
doubled seams from the outer circle, and $c$ is the doubled seam
between the inner circles. The map $j$ exchanges the two halves.
Thus it reverses $a,c,e$ and fixes $b,d$ pointwise.

Choose annular twist representatives satisfying
$jT_aj=T_a^{-1}$ and $jT_bj=T_b^{-1}$, and define the right-hand
twists by $\phi$ conjugacy. These choices can be made simultaneously.
In annular coordinates an orientation-reversing involution is of the
form $(\theta,t)\mapsto(-\theta+c,t)$ or
$(\theta,t)\mapsto(\theta+c,-t)$. A twist
$(\theta,t)\mapsto(\theta+h(t),t)$ is reversed in the first case
for any profile, and in the second if $h(-t)=1-h(t)$. Integer end
values make it the identity off the collar. Choose the $T_a$ profile
zero on a smaller core collar, and make the same choice for $T_e$
by conjugacy. These choices give identities of maps
\begin{equation}\label{eq:reflection-identities}
 \epsilon^2=1,\qquad \epsilon\psi\epsilon=\chi,
 \qquad \epsilon\chi\epsilon=\psi.
\end{equation}
They preserve both half-turn conjugacies and the fixed neighborhoods
of the two section points. Average a metric over the finite group
generated by $\phi,\epsilon$ and use its area form to choose these
annular coordinates and twists. The orientation-preserving finite
symmetries and the twists preserve area, while the reflections reverse
it. No metric invariant under the Dehn twists is required.

\subsection{The boundary map with both seams retained}
Realize the $\delta$ mapping torus by two pieces
$F\times[0,1]\times\{0,1\}$, glued by
\[
 (z,1,0)\sim(\chi z,0,1),\qquad
 (z,1,1)\sim(\psi z,0,0).
\]
The map $(z,u,i)\mapsto(\epsilon z,u,1-i)$ respects both seams
by \eqref{eq:reflection-identities}. It is a free involution because
it covers a half-rotation of the base circle; it reverses the fiber
orientation and preserves each section circle.

For a formula in a single mapping torus, write
$Y_\delta=(F\times[0,2])/((z,2)\sim(\delta z,0))$.
The two pieces map to $[z,u]$ and $[\chi^{-1}z,1+u]$.
On $F\times\mathbb R$ the lifted boundary map is
\[
                  (z,t)\longmapsto(\chi^{-1}\epsilon z,t+1).
\]
Its square is $(z,t)\mapsto(\delta^{-1}z,t+2)$, the deck
transformation, since $(\chi^{-1}\epsilon)^2=\delta^{-1}$.
Transport through $F_U:Y_\rho\to Y_\delta$, $[z,t]\mapsto[Uz,t]$,
to define the boundary involution $\sigma$ on $Y=Y_\rho$.
Its lifted formula is
\begin{equation}\label{eq:sigma}
 (z,t)\longmapsto(f_\rho z,t+1),\qquad
 f_\rho=U^{-1}\chi^{-1}\epsilon U,\qquad f_\rho^2=\rho^{-1}.
\end{equation}
The factor $\chi^{-1}$ is part of the formula; reflection alone in
a one-seam model would not describe this involution.

\begin{lemma}\label{lem:frames}
The parallel normal framing along the boundary of either proper
section extends over that section. In these framings $D\sigma$ on
the normal planes is a constant reflection. Each doubled section
in $Z$ has self-intersection zero.
\end{lemma}
\begin{proof}
The vertical normal bundle of a section is a flat oriented plane
bundle with holonomies $D\phi=-I$ and $D\psi=I$. Its boundary
holonomy is $I$. Cap the section by a flat disk using the parallel
boundary frame. The resulting flat plane bundle on a torus has
Euler number zero: the holonomies $-I,I$ are realized on the trivial
bundle by the flat connection $\pi J_0\,du$, where $J_0$ is the
infinitesimal rotation and $u$ is the first circle coordinate.
Thus the original relative Euler number is zero and the frame extends.

The maps $\chi,\psi,U$ are the identity near the section points.
Equation~\eqref{eq:sigma} therefore has normal derivative
$D\epsilon$, the same constant reflection on both halves of the
boundary circle. To form the oriented double, reverse both the fiber
and base orientations on its second copy. The four-dimensional
orientation is unchanged. Identifying its normal plane by a fixed
reflection makes the clutching map a constant element of
$GL^+(2,\mathbb R)$. Its winding is zero. The relative Euler numbers
on both sections are zero as well, so the doubled normal Euler number
is zero. Surgery takes place away from the sections and does not
change this calculation.
\end{proof}

\section{Extension over the filling and the group of the double}
\label{sec:extension}

\begin{proposition}\label{prop:extension}
The map $\sigma$ extends to an orientation-reversing involution
$I:V\to V$. The retraction and fundamental-group assertions of
Theorem~\ref{thm:geometric}(4) follow.
\end{proposition}
\begin{proof}
Use the representatives just constructed and put $\eta=T_aT_ej$.
Twists on opposite handles commute; $j$ inverts each twist; and
$\phi$ exchanges $T_a,T_e$. Hence
\begin{equation}\label{eq:eta}
       \eta^2=1,\qquad \eta\phi=\phi\eta,
       \qquad \eta\psi\eta^{-1}=\psi^{-1}.
\end{equation}
On the corner-excised bundle define
\begin{equation}\label{eq:I}
                   (z,v,w)\longmapsto(\eta z,1-v,1-w).
\end{equation}
The identities $\phi\eta\phi=\eta$ and $\psi\eta\psi=\eta$
give compatibility at the two seams. Equal corner excisions and
equivariant smoothing make this a smooth involution of $R$.
It reverses four-dimensional orientation and preserves both surgery
tori. Near $c,e$, the map $\eta$ agrees with $j$: the supports of
$T_a,T_e$ miss $c$, and the $T_e$ profile is zero near the core $e$.
No disjointness from $c$ is asserted for the supports of $T_b,T_d$.

In suitable area coordinates on these collars,
$\eta(\theta,t)=(-\theta+c_i,t)$. At alpha, substitution in
\eqref{eq:charts} gives
\[
 Q'_1=-Q_1+c_\alpha-\tfrac12,\quad Q'_2=-Q_2-1,
 \qquad
 \begin{pmatrix}Y'_1\\Y'_2\end{pmatrix}
 =\begin{pmatrix}1&0\\-1&-1\end{pmatrix}
   \begin{pmatrix}Y_1\\Y_2\end{pmatrix}.
\]
At beta the corresponding formula is
\[
 Q'_1=-Q_1+c_\beta,\quad Q'_2=-Q_2-1,
 \qquad(Y'_1,Y'_2)=(Y_1,-Y_2).
\]
Both torus directions reverse, and each constant normal matrix has
determinant $-1$. There is no longitude-dependent framing turn.
With a common basing path transported for each triple, the induced
peripheral maps are therefore
\[
 (M,C,\ell)\mapsto(M^{-1},C^{-1},\ell^{-1}),\qquad
 (N,E,D)\mapsto(N^{-1},E^{-1},D^{-1}).
\]
We next extend the involution itself, including its order-two property.
Use invariant elliptical normal disks for the two constant normal
involutions above, and choose the angular origin on a fixed axis.
In longitude and meridian coordinates the boundary action is then
\[
 (q_1,q_2,\theta)\longmapsto
       (-q_1+b_1,-q_2+b_2,-\theta),
       \qquad q_1,q_2,\theta\in\mathbb R/\Z.
\]
The ordered peripheral bases are $(C,\ell,M)$ and $(E,D,N^{-1})$.
In either basis the selected slope is meridian plus second longitude.
Use the linear gluing map
\[
 S(q_1,q_2,\theta)=(q_1,q_2+\theta,\theta).
\]
It fixes the old longitudes and sends the inserted meridian to the
specified slope. It commutes with the displayed affine involution.
Consequently the cap map
\[
 (q_1,q_2,z)\longmapsto(-q_1+b_1,-q_2+b_2,\overline z)
       \quad\text{on }T^2\times D^2
\]
glues to \eqref{eq:I} and squares to the identity.

This linear filling is diffeomorphic, with the outer boundary fixed,
to the specified symplectic filling. Indeed the boundary shear of
\cite[Section 2.1]{ADK03} has the form
$(q_1,q_2,\theta)\mapsto(q_1,q_2+h(\theta),\theta)$, where
$h:S^1\to S^1$ has degree one. Interpolating a lift of $h$ to
$\theta$ gives an isotopy of gluing maps to $S$, preserving the
longitude classes and the meridian image. Changing the shape of the
normal disk gives the same marked isotopy. Extend this isotopy over
an exterior collar, and conjugate the involution of the linear filling
by the resulting diffeomorphism. Conjugation preserves its square,
so we obtain an involution on the specified filling, not merely an
extension of its peripheral homology action.

It remains to identify its outer boundary restriction with the
specified $\sigma$. Start at the south end of the rounded corner
$(1,1)$. In the induced boundary direction the corners occur in the
order $(1,1),(1,0),(0,0),(0,1)$, with seam crossings
$\psi,\phi,\psi^{-1},\phi^{-1}$. The total transport is $\rho$;
the first half transports by $\phi\psi$. The base rotation in
\eqref{eq:I} shifts two corner arcs, so its fiber factor in the
continuous boundary marking is $(\phi\psi)^{-1}\eta$. This equals
the factor in \eqref{eq:sigma}, since
\begin{align*}
 U^{-1}\chi^{-1}\epsilon U
 &=\psi^{-1}T_e\chi^{-1}\epsilon T_e^{-1}\psi\\
 &=\psi^{-1}T_e\chi^{-1}T_a\chi\epsilon
  =\psi^{-1}T_aT_e\epsilon
  =\psi^{-1}\phi\eta.
\end{align*}
Here $\epsilon T_e^{-1}\epsilon=T_a$ and $T_a$ commutes with
the right-hand twists. Thus $I|_Y=\sigma$ for the chosen representatives.

For completeness, a change from another symmetric $T_a$ profile can
be made without changing the surgery marking. Keep $T_b$ fixed and
use $\psi_u=T_a(u)T_b$. The target-side isotopy
$\psi_u\psi_0^{-1}=T_a(u)T_a(0)^{-1}$ is supported near $a$.
Insert it near the vertical seam, away from $w=1/2$; its fiber support
misses $e$ and the section points. The resulting bundle identification
fixes both surgery core collars. The $T_e$ change belongs to the
boundary marking, and its induced boundary isotopy extends over an
outer collar. Hence the extension statement is unchanged by these
representative choices.

Now fold $Z=V_1\cup_\sigma V_2$ onto $V_1$ by the identity on
$V_1$ and $I^{-1}$ on $V_2$. On a seam point $b\sim\sigma(b)$
the two values agree. This is a retraction, so $G$ injects into
$\pi_1(Z)$. Changing the second marking by $I^{-1}$ makes the
two boundary homomorphisms identical. Their van Kampen pushout
factors through their common image $H_\partial$ and is
$G*_{H_\partial}G$. Both amalgamating maps are inclusions by the
definition of $H_\partial$. It follows that $\pi_1(Z)=1$ exactly
when $G=1$.
\end{proof}

\section{Homology, symplectic surfaces, and the closed quotients}
\label{sec:topology}

\subsection{Integral homology of the piece and double}
We use the complete presentation proved in
Proposition~\ref{prop:fifteen}; its derivation is independent of the
present section. We label a transport relation by its stable letter
and its underlying fiber generator: for example, $(Bs)$ means
$B(K_0^{-1}s)B^{-1}=s$, and $(By)$ means
$B(M^{-1}y)B^{-1}=yx$. These are labels for equations, not group
words. In the abelianization, $(Bs)$ gives $M=0$, and $(By)$ gives $x=-M=0$.
The $Bx$ relation gives $y=-x=0$. The $Ax,Ay$ relations then give
$r=s=0$, and the $As$ relation gives $N=0$. The two filling relations
finally give $A=B=0$. These are integral eliminations with unit
coefficients, proving $H_1(V)=0$.

The original bundle has Euler characteristic
$\chi(F)\chi(\Sigma_{1,1})=2$, unchanged by torus surgery.
The boundary is connected. In the pair sequence for $(V,Y)$, the
boundary map $H_4(V,Y)\to H_3(Y)$ is an isomorphism and
$H_3(V,Y)=H^1(V)=0$, so $H_3(V)=0$. The same sequence gives
$H_1(V,Y)=0$, hence $H^3(V)=0$. Universal coefficients imply that
$H_2(V)$ is torsion free, and Euler characteristic gives its rank
as one. A surviving fiber meets $\Sigma$ once, so it generates
$H_2(V)$ primitively, and $\Sigma$ generates $H_2(V,Y)$.
Since $F^2=0$, the map $H_2(V)\to H_2(V,Y)$ is zero. Thus
\begin{equation}\label{eq:homology-maps}
 H_2(Y)\xrightarrow{\cong}H_2(V),\qquad
 H_2(V,Y)\xrightarrow{\partial\,\cong}H_1(Y).
\end{equation}
Both boundary groups are $\Z$, generated by the fiber and the
knot meridian, respectively. The mod-two second homology of $V$
is generated by $F$, and $w_2(V)$ evaluates on it as $F^2=0$.
Hence $V$ is spin and, since $H^1(V;\Ftwo)=0$, has a unique spin
structure.

Mayer--Vietoris gives $H_1(Z)=0$ and $\chi(Z)=4$. Integral duality
then gives $H_2(Z)=\Z^2$. The fiber and either doubled section
have genus two, and Lemma~\ref{lem:frames} gives their matrix
\begin{equation}\label{eq:hyperbolic}
                          \begin{pmatrix}0&1\\1&0\end{pmatrix}.
\end{equation}
Its determinant has absolute value one, so these classes form a basis.
This proves the cohomology-ring assertion and shows that $Z$ is spin.

\subsection{Symplectic representatives and minimum genus}
Before surgery, $R\cup_\sigma R$ is an oriented genus-two surface
bundle over a genus-two surface. As in
\cite[Theorem 8 and footnote 3]{LP25}, reverse both fiber and base
orientations on the second copy to obtain the oriented bundle
description. A fiber area form with a sufficiently large base area
term gives a symplectic form \cite{Thu76}. The fixed sections are
symplectic, the suspension tori are Lagrangian, and a fiber away from
the surgery circles is symplectic. The local Luttinger surgeries on
both copies preserve the form near that fiber and the sections.
Thus $F,H$ are symplectic in $Z$.

\begin{proposition}\label{prop:genus}
If $S=F$ or $H$ and $k\ne0$, the minimum genus of $k[S]$ in $Z$
is $|k|+1$. These are all the nonzero square-zero classes.
\end{proposition}
\begin{proof}
A symplectic tubular neighborhood of $S$ can be identified with
$S\times D^2$ with product symplectic form, since its normal Euler
number is zero. Choose a smooth map $h:S\to S^1$ inducing a
surjection on fundamental groups. For $k>0$ and sufficiently small
$\varepsilon>0$, define
\[
 S_k=\{(u,z)\in S\times D^2:z^k=\varepsilon h(u)\}.
\]
It is a connected unbranched $k$-fold cover of $S$: the monodromy
cyclically permutes all roots. It represents $k[S]$ and has genus
$k+1$. Since $|z|$ is constant, the disk area form restricts to zero,
whereas the surface area form is positive. Thus $S_k$ is symplectic.
The symplectic Thom theorem \cite[Theorem 1.1]{OS00} makes it genus
minimizing. Orientation reversal gives the assertion for negative
$k$. Finally $(a[F]+b[H])^2=2ab$, proving the description of all
square-zero classes.
\end{proof}

In particular $Z$ has no homologically nonzero square-zero sphere or
torus. The essential square-zero spheres in $S^2\times S^2$ exclude
a diffeomorphism with $Z$. This argument uses no simple-connectivity
hypothesis.

\subsection{Proper surfaces and handles in the filling}
The minimum-genus calculation also constrains the piece $V$ itself.
This consequence is useful when trying to turn a formal group identity
into a smooth handle cancellation.

\begin{proposition}\label{prop:relative-genus}
Let $S\subset V$ be a compact connected oriented properly embedded
surface of genus $g$ with $b\ge1$ boundary components. If
$m=[S]\cdot[F]\ne0$, then
\[
                         2g+b-2\ge |m|.
\]
Every proper oriented disk or annulus represents zero in
$H_2(V,Y;\Z)$. The inequality is sharp for either surviving
punctured-torus section. Moreover, $V$ admits no smooth handle
decomposition consisting only of handles of indices $0,1,2$.
\end{proposition}
\begin{proof}
The extension $I$ identifies the ordinary oriented double
$D(V)=V\cup_{\mathrm{id}}(-V)$ with $Z$, preserving its first copy.
Straighten $S$ in a boundary collar and double it. The resulting
oriented surface $DS$ has genus $2g+b-1$. Its normal bundle has a
nonzero section: the normal bundle on $S$ is trivial because
$H^2(S;\Z)=0$, and a section constant along the collar glues to its
reflection on the other half. Thus $[DS]^2=0$.
Keep $F$ in the first copy of $V$. Then $[DS]\cdot[F]=m$.
In the hyperbolic basis $F,H$, these two equations force
$[DS]=m[H]$ when $m\ne0$. Proposition~\ref{prop:genus} gives
$2g+b-1\ge |m|+1$.

Duality identifies $H_2(V,Y;\Z)$ with
$\operatorname{Hom}(H_2(V;\Z),\Z)$, so intersection with $F$
detects its classes. The inequality excludes a nonzero class for
a disk or annulus. A section has $(g,b,|m|)=(1,1,1)$ and realizes
equality.

If a handle decomposition used only indices $0,1,2$, its $2$-handle
cocores would be proper disks. Turn the decomposition upside down:
the relative handles have indices $2,3,4$, so the cocores generate
$H_2(V,Y)=C_2(V,Y)/\operatorname{im}\partial_3$. All their classes
would vanish, contradicting $H_2(V,Y)=\Z$.
\end{proof}

This proposition concerns embedded surfaces. A nullhomotopy of $J$
is a mapped disk and need not be embedded. Resolving its intersections
can increase its genus. Thus the proposition neither excludes such a
nullhomotopy nor contradicts hypothetical simple connectivity of $V$.
Also $H_2(V)\to H_2(V,Y)$ is zero, so adding an absolute spherical
class does not alter the relative intersection number $m$.

\subsection{The boundary quotient and sliceness}
The double cover gives an exact sequence
\[
 1\longrightarrow\pi_1(Z)\longrightarrow\pi_1(W)
                   \longrightarrow\Z/2\longrightarrow1.
\]
Since $H_1(Z)=0$, its fundamental group is perfect, and its image
vanishes in any abelian quotient of $\pi_1(W)$. Hence
$H_1(W)=\Z/2$. Also $\chi(W)=2$. Duality shows that $W$ is a
rational homology sphere, with integral groups $H_2(W)=\Z/2$ and
$H_3(W)=0$, and $\dim H_2(W;\Ftwo)=2$.

To check spin, consider the image $S$ of $\Sigma$ in $W$.
Its boundary is identified by a free orientation-preserving half-rotation,
so $S$ is a closed nonorientable surface. In the extending normal
frame of Lemma~\ref{lem:frames}, choose a nonzero vector in the
$+1$ eigenspace of the constant reflection $D\epsilon$. Extend
this vector using the frame over $\Sigma$. It matches at the seam
and gives a nowhere-zero normal field on $S$. Thus $S^2=0$ over
$\Ftwo$. The image $\overline F$ of an interior fiber also has
square zero and meets $S$ once. They form a hyperbolic basis of
$H_2(W;\Ftwo)$, and the Wu formula gives $w_2(W)=0$.
The framing calculation, rather than nonorientability alone, is what
establishes the square of $S$.

Suppose that the figure-eight knot $4_1$ bounded a smooth disk
$D_0$ in punctured $W$. If its relative integral class were zero,
its mod-two class would be zero, so the disk would be characteristic
in this spin manifold. The relative Rochlin formula
\cite[Theorem 2]{Klu21}, with signature zero, self-intersection
zero, and boundary $S^3$, would give $\operatorname{Arf}(4_1)=0$,
contrary to $\operatorname{Arf}(4_1)=1$. This formula does not
require the first homology of the four-manifold to vanish.

If the disk class were nonzero, its neighborhood together with the
puncture ball would give an embedding of the zero trace $X_0(4_1)$
in $W$ inducing a nonzero map on $H_2$. The framing is zero because
the integral intersection pairing of $W$ vanishes. The trace is
simply connected, so the embedding lifts to $Z$. A core capped by
a genus-one Seifert surface gives a square-zero torus in the lifted
trace. Its homology class is nonzero, since its push-forward to $W$
is nonzero, even though the latter class is torsion. This contradicts
Proposition~\ref{prop:genus}. Thus $4_1$ is not slice in $W$.
This is the slicing obstruction of \cite{LP25}, with its normal
framing and homological hypotheses made explicit for our filling.

\subsection{The odd boundary regluing}\label{subsec:odd}
We use two specific inputs from \cite[Section 3]{LP25}.
First, the boundary map $f_0$ in their Figure 3 sends the framed
meridian $(\mu,0)$ to $(\mu',1)$, where $(\mu',0)$ is framed
homologous to $(\mu,0)$; framings are measured relative to the
Seifert framing in that diagram. Transport this diffeomorphism to
$Y$ by the boundary identification of Section~\ref{sec:boundary},
so $f=F_U^{-1}f_0F_U$ in those markings. This is orientation
preserving, as is $f_0$; conjugation by the marking also preserves
its permutation of the two spin structures. The framing property
implies that $f$ interchanges the two spin structures on $Y$:
a single framing twist changes the induced spin structure on the
meridian, which generates $H_1(Y;\Ftwo)$.

The unique spin structures on the two copies of $V$ match under
$\sigma$, because $Z$ is spin. They therefore fail to match under
$f\circ\sigma$. The resulting $Z_f$ is non-spin. As above,
$H_1(Z_f)=0$, $\chi(Z_f)=4$, and $H_2(Z_f)=\Z^2$.
Novikov additivity gives signature zero, since the form of $V$
has zero signature. The unimodular rank-two form is consequently
the odd form $\langle1\rangle\oplus\langle-1\rangle$.

Second, we use the relative Heegaard Floer invariants and their
pairing from \cite[Lemmas 9--10]{LP25}, with coefficients $\Ftwo$.
On $Y$ the two structures $\mathfrak s_\pm$ with
$\langle c_1(\mathfrak s_\pm),[F]\rangle=\pm2$ satisfy
$HF_{\mathrm{red}}(Y,\mathfrak s_\pm)=\Ftwo$.
The required relative nonvanishing needs a separate argument:
the symplectic double $Z$ has $b^+=1$, whereas the canonical
nonvanishing theorem used below requires $b^+>1$.

\begin{lemma}\label{lem:relative-nonvanishing}
The relative invariants of $V$ in the extensions of
$\mathfrak s_\pm$ are nonzero. This assertion does not require $G=1$.
\end{lemma}
\begin{proof}
Write $Z=V_1\cup_\sigma V_2$ and choose a square-zero symplectic
genus-two fiber in the interior of $V_1$, disjoint from the surgery
neighborhoods. Form a symplectic sum of two copies of $Z$ along
these fibers \cite{Gom95}. Denote the result by $Z_*$, retaining
the untouched $V_2$ from the first copy together with a collar of
its boundary $Y$. The sum formulas give
\[
 \chi(Z_*)=2\chi(Z)-2\chi(F)=12,\qquad
 \sigma(Z_*)=0,\qquad b^+(Z_*)=5+b_1(Z_*)>1.
\]
For its canonical Spin$^c$ structure $\mathfrak k$, the canonical
nonvanishing theorem \cite[Theorem 1.1]{OS04} gives
$\Phi_{Z_*,\mathfrak k}(1)=1$ over $\Ftwo$.

Let $F_Y\subset Y$ be a fiber oriented by the symplectic form.
Adjunction gives $\langle c_1(\mathfrak k),[F_Y]\rangle=-2$.
Its restriction to $Y$ is therefore non-torsion and is one of
$\mathfrak s_\pm$ in the original marking. Remove a ball from
$V_2$ and one from its complement to obtain a cobordism
$W=W_1\cup_Y W_2$ from $S^3$ to $S^3$, with
$W_1=V_2\setminus B^4$. Put $\mathfrak k_i=\mathfrak k|_{W_i}$.
For a non-torsion restriction on the cut,
the composition formula \cite[Section 2]{OS04} is
\begin{equation}\label{eq:auxiliary-floer-cut}
 F^+_{W_2,\mathfrak k_2}\,\Pi_Y^{\mathrm{red}}
 F^-_{W_1,\mathfrak k_1}(\Theta^-)
 =\sum_{\substack{\mathfrak s|_{W_i}=\mathfrak k_i\\i=1,2}}
                  F^{\mathrm{mix}}_{W,\mathfrak s}(\Theta^-).
\end{equation}
Here $\Theta^\pm$ are the standard generators for $S^3$, and
$\Pi_Y^{\mathrm{red}}$ is the canonical map through reduced Floer
homology. This formula requires $b^+(W)>1$, but does not require
positive $b^+$ on both sides of this non-torsion cut.

We check that the coefficient of $\Theta^+$ on the right cannot
cancel. By Mayer--Vietoris the extensions in the sum differ by
the image of $H^1(Y;\Z)=\Z$. Its generator maps to
$\operatorname{PD}_{Z_*}[F_Y]$. The latter has infinite order,
since $F_Y$ has positive symplectic area. Thus the extensions are
$\mathfrak k+n\operatorname{PD}[F_Y]$, $n\in\Z$.
Writing
\[
 d(\mathfrak s)=
 \frac{c_1(\mathfrak s)^2-2\chi(Z_*)-3\sigma(Z_*)}{4},
\]
we have $d(\mathfrak k)=0$ and, since $F_Y^2=0$,
\[
 d\bigl(\mathfrak k+n\operatorname{PD}[F_Y]\bigr)=-2n.
\]
The grading of the mixed map therefore permits a $\Theta^+$
coefficient only for $n=0$. That coefficient is
$\Phi_{Z_*,\mathfrak k}(1)=1$. Hence the left side of
\eqref{eq:auxiliary-floer-cut} is nonzero, and in particular the
relative invariant
$\Pi_Y^{\mathrm{red}}F^-_{W_1,\mathfrak k_1}(\Theta^-)$
is nonzero. Conjugation symmetry \cite[Theorem 3.6]{OS06} gives
the other boundary structure. Finally $H^2(V;\Z)=\Z$, with
restriction an isomorphism to $H^2(Y;\Z)$, so each of these
boundary structures has a unique extension to $V$.
\end{proof}

Every self-diffeomorphism of $Y$ permutes $\mathfrak s_+,
\mathfrak s_-$. The one-dimensional pairing is nonzero on nonzero
elements, so compatible extensions on the two copies give a nonzero
mixed invariant of $(Z_f,\mathcal L,\mathfrak t)$ for some
$\mathfrak t$, where $\mathcal L\subset H_2(Z_f;\Q)$ is the
line spanned by the boundary fiber. This line is nonzero: by
\eqref{eq:homology-maps} the Mayer--Vietoris map into the two copies
of $H_2(V;\Q)$ is diagonal up to signs, and either fiber survives
in its cokernel.

The line must be retained when $b^+=1$. In
$\CP^2\#\bCP^2$ the only rational square-zero lines are spanned
by $h+e$ and $h-e$, where $h^2=1$, $e^2=-1$, and $h\cdot e=0$.
Both have embedded sphere representatives. Their tubular boundaries
are separating copies of $S^2\times S^1$ realizing these lines.
The reduced Floer groups vanish in the relevant non-torsion
structures, so the corresponding mixed invariants vanish for either
line. This is the comparison in \cite[proof of Theorem 2]{LP25}.
It excludes a diffeomorphism with $Z_f$ without a hypothesis on
its fundamental group.

\subsection{The conditional classifications}
\begin{proof}[Proof of Corollary~\ref{cor:conditional}]
If $G=1$, van Kampen gives $\pi_1(Z)=\pi_1(Z_f)=1$. Freedman's
classification \cite{Fre82}, applied to their respective forms and
their vanishing Kirby--Siebenmann invariants, gives the asserted
homeomorphisms. The preceding smooth obstructions distinguish both
from their standard models.

If $\pi_1(Z)=1$, the covering sequence gives $\pi_1(W)=\Z/2$.
The manifolds $W$ and $\BK$ then have the same finite cyclic
fundamental group, zero intersection form on $H_2/\mathrm{torsion}$,
spin $w_2$-type, and Kirby--Siebenmann invariant zero.
Hambleton--Kreck's classification \cite[Theorem C]{HK93} gives
a homeomorphism. That theorem has no lower bound on Euler
characteristic. Conversely, if $W$ is homeomorphic to $\BK$,
its connected double cover corresponds to the trivial subgroup
of $\Z/2$, so $\pi_1(Z)=1$. Proposition~\ref{prop:extension}
identifies this condition with $G=1$. The figure-eight knot is
slice in $\BK$ and not in $W$, so they are not diffeomorphic.
\end{proof}

Together these arguments prove Theorem~\ref{thm:geometric}.
They leave the hypothesis in Corollary~\ref{cor:conditional}
unresolved.

\section{The involution orbit space}\label{sec:orbit}
There is a second quotient associated to the boundary extension.
Choose the involutive representatives and product filling extensions
used in Proposition~\ref{prop:extension}, and define
\[
 \mathcal F=\operatorname{Fix}(I),\qquad
 \mathcal O=V/\langle I\rangle,\qquad
 \mathcal O^{\mathrm{reg}}=\mathcal O\setminus p(\mathcal F).
\]
Here $p$ is the orbit projection. Unlike $W$, this quotient identifies
interior points as well. It is a singular space, and its regular part
is the complement of its singular stratum, with the outer boundary
retained.

\begin{proposition}[The orbit space and its regular part]\label{prop:orbit}
The fixed set $\mathcal F$ is one interior circle, and $\mathcal O$
is contractible. If $\Lambda=\pi_1(\mathcal O^{\mathrm{reg}})$, there
is a split exact sequence
\begin{equation}\label{eq:orbit-sequence}
 1\longrightarrow G\longrightarrow\Lambda
   \longrightarrow\Z/2\longrightarrow1.
\end{equation}
The image $\iota$ of the nontrivial loop in a normal
$\mathbb{RP}^2$ link has order two and normally generates $\Lambda$.
Furthermore $[\Lambda,\Lambda]=G$. In particular,
$G=1$ if and only if $\Lambda=\Z/2$.
\end{proposition}
\begin{proof}
The base half-turn has three fixed points on the punctured torus:
the fourth fixed point of the closed torus was removed at the puncture.
The section through $p_0$ is invariant and contains a fixed point over
each. On a fixed fiber the return map reverses orientation, so its
nonempty fixed set is one-dimensional. In a filling cap the action
inverts both torus coordinates and reflects the disk. Its fixed set is
four intervals. These join the exterior fixed set across the equivariant
collars. Since the boundary action is free, $\mathcal F$ is a nonempty
union of $k$ interior circles.

The section quotient is an oriented surface with one boundary
component and three simple branch points. Its Euler characteristic
satisfies $-1=2\chi-3$, so it is a disk. The two section generators
$A,B$ therefore vanish under $p_*$. This map is onto on fundamental
groups: use an admissible equivariant triangulation \cite{Ill78},
base at a fixed vertex, and lift an edge loop one edge at a time,
translating a lift by $I$ when necessary. Its final vertex must be
the starting vertex, since that vertex has a singleton orbit.
In the presentation of Proposition~\ref{prop:fifteen}, imposing
$A=B=1$ gives $r=x$, $s=y$, then $K_0=M=1$ from the $Bs$ row,
$x=1$ from the alpha slope, $y=1$ from $Bx$, and $N=1$ from
the beta slope. Thus $A,B$ normally generate $G$, and
$\pi_1(\mathcal O)=1$.

Write $S_I=p(\mathcal F)$. With coefficients $\Ftwo$, the
admissible triangulation gives an exact sequence of chain complexes
\[
 0\longrightarrow C_*(\mathcal O,S_I)
 \xrightarrow{\mathrm{tr}}C_*(V)
 \xrightarrow{p_*}C_*(\mathcal O)\longrightarrow0.
\]
The transfer sends a nonfixed cell to the sum of its two lifts;
fixed faces cancel in their boundary. On fixed cells $p_*$ is the
identity. Since $S_I$ is a union of circles, the pair sequence
identifies $H_j(\mathcal O,S_I)$ with $H_j(\mathcal O)$ for
$j\ge3$. Using these identifications and the dimension bound,
the associated transfer homology sequence first gives
$H_4(\mathcal O)=H_3(\mathcal O)=0$ and then an injection
$H_2(\mathcal O,S_I)\hookrightarrow H_2(V)=\Ftwo$.
The pair sequence, using simple connectivity, gives
\[
 0\longrightarrow H_2(\mathcal O)
 \longrightarrow H_2(\mathcal O,S_I)
 \longrightarrow\Ftwo^k\longrightarrow0.
\]
Consequently $k=1$ and $\mathcal O$ is acyclic over $\Ftwo$.

The involution acts by $-1$ on $H_2(V;\Z)$: this follows from
fiber-orientation reversal on the boundary and the isomorphism
$H_2(Y)\to H_2(V)$. Over a field of characteristic different from
two, the orbit chain complex is the coinvariant complex, and taking
coinvariants is exact by averaging. The only positive-dimensional
homology class of $V$ is negated, so $\mathcal O$ is acyclic over
every such field as well. Finite generation and universal coefficients
give integral acyclicity. Hurewicz and Whitehead now give contractibility.

Removing a smooth circle from a four-manifold does not change its
fundamental group: general position moves loops and their homotopies
off the circle. The free double cover
$V\setminus\mathcal F\to\mathcal O^{\mathrm{reg}}$ gives
\eqref{eq:orbit-sequence}. The oriented rank-three normal bundle of
the fixed circle is trivial, and $I$ acts by $-1$ on each normal
fiber. The quotient neighborhood is therefore
$S^1\times\operatorname{cone}(\mathbb{RP}^2)$.
The nontrivial normal-link loop has order exactly two in $\Lambda$:
its square dies in the link, whereas its covering monodromy is
nontrivial in $S^2\to\mathbb{RP}^2$. It supplies the splitting.
Reattaching the singular neighborhood kills precisely this loop in
van Kampen, giving $\pi_1(\mathcal O)=\Lambda/\nc\iota$.
Finally $G$ is perfect and $\Lambda/G$ is abelian, whence
$[\Lambda,\Lambda]=G$.
\end{proof}

Contractibility of $\mathcal O$ does not compute $\Lambda$.
A nullhomotopy downstairs can meet the singular circle, where the
transverse model is $\operatorname{cone}(\mathbb{RP}^2)$ rather
than $\mathbb R^3$. The general-position argument upstairs therefore
does not apply downstairs. Nor does normal generation by an involution
make a group cyclic. The remaining question in this description is
whether $\Lambda=\langle\iota\rangle$.

\section{Other unit fillings and changes of description}\label{sec:variants}
The fixed target is not canonical among the possible parametrizations.
It is therefore useful to separate a different description of $V$
from a different filling with the same elementary invariants.

\begin{proposition}\label{prop:variants}
For $p,q\in\{1,-1\}$ let $V_{p,q}=W(p,q;0,0)$, with the
same cotangent charts and outer boundary marking. Each $V_{p,q}$
is a spin integral homology $S^2$. Its double has the symplectic
hyperbolic basis and minimum-genus formula of
Theorem~\ref{thm:geometric}(1). The spin boundary quotient, slicing
obstruction, odd-regluing construction, and conditional classification
arguments apply to each filling. The boundary involution extends over
each of them, so its double is simply connected exactly when the piece
is. No one of these four groups is decided here.
\end{proposition}
\begin{proof}
The complement relations are unchanged. In abelianization they kill
$M,N,x,y,r,s$, and the filling rows impose $pA=qB=0$. More generally,
for integral $p,q$ and these fixed sections,
\[
 H_1(W(p,q;0,0);\Z)=\Z/p\oplus\Z/q,
 \qquad \Z/0=\Z.
\]
For unit coefficients this is zero. The homology, spin, and boundary
homology calculations of Section~\ref{sec:topology} then apply.
In the cotangent convention the Luttinger coefficients are $p,-q$;
both signs are allowed in the local construction \cite{ADK03}.
The fiber and sections miss the surgery neighborhoods, so their
symplectic representatives and framings persist. The same double
therefore has the hyperbolic basis and the same minimum-genus argument.
The quotient spin calculation uses the unchanged section framing.
The auxiliary symplectic sum and grading argument of
Lemma~\ref{lem:relative-nonvanishing} apply to each $V_{p,q}$:
its symplectic double contains the same untouched fiber and copy of
$V_{p,q}$, and the boundary and cohomology calculations are unchanged.
Thus the relative nonvanishing needed for the odd regluing remains
valid. Finally the
peripheral action of $I$ is $-\mathrm{id}$ on each whole triple,
so it preserves every chosen unoriented slope. The cap construction
of Proposition~\ref{prop:extension}, with
$S(q_1,q_2,\theta)=(q_1,q_2+k\theta,\theta)$ for the relevant
$k=\pm1$, extends it as an involution over every such filling.
\end{proof}

\subsection{A reflection and the section index}
There is a useful reflection of the marked exterior. Its peripheral
dictionary, with a common conjugator for each triple, is
\begin{equation}\label{eq:reflection-dictionary}
\begin{aligned}
 (M,C,\ell)&\longmapsto A(M^{-1},C^{-1},\ell C)A^{-1},\\
 (N,E,D)&\longmapsto w(N,E^{-1},D^{-1})w^{-1},
                  &w&=M^{-1}r.
\end{aligned}
\end{equation}
Here a parenthesized triple means that each entry is conjugated.
For an algebraic check, the following formulas define an involution
of the complement group (using $[M,C]=1$, not of the free group):
\[
\begin{array}{c|cccccccc}
g&x&y&r&s&M&N&A&B\\ \hline
g'&x^{-1}&xM^{-1}y&r^{-1}&M^{-1}rs&xM^{-1}x^{-1}
 &M^{-1}rNr^{-1}M&A&B^{-1}.
\end{array}
\]
The geometric reflection has the same peripheral action. With the
representatives of Section~\ref{sec:extension}, it is induced by
$(z,v,w)\mapsto(\eta z,v,1-w)$. At alpha it reverses the
fiber direction and preserves the base direction; the half-drift
section acquires one fiber winding. Thus the longitude becomes
$\ell C$ and the normal meridian is reversed. At beta both torus
directions reverse and the normal meridian is preserved. More explicitly, in the alpha coordinates
$\Theta=\theta-v/2$, $P_1=t$, $P_2=t/2-\kappa(w-1/2)$, the map is
\[
 (\Theta,v,P_1,P_2)\longmapsto
 (-\Theta-v+\theta_0,v,P_1,P_1-P_2).
\]
The increments $(1,0)$ of $C$ become $(-1,0)$; those $(0,-1)$
of $\ell$ become $(1,-1)$. At beta the angular coordinates change
by $(\theta,w)\mapsto(-\theta+\theta_1,1-w)$ and the momenta are
fixed. Constant momenta remain constant in both charts, proving
preservation of the Lagrangian framing and the peripheral action. This argument uses the peripheral map;
it does not identify the two descriptions on every based generator.

Applying the dictionary to the whole slopes and inverting the first
image to normalize its meridian gives
\begin{equation}\label{eq:reflection-parameters}
                (p,q;n,m)\longmapsto(-p,-q;1-n,m).
\end{equation}
In particular our filling is carried to $W(-1,+1;1,0)$, whose
alpha slope is $M(\ell C)^{-1}$. This does not identify it with
$W(-1,+1;0,0)$. Distinct slopes in the exterior do not, conversely,
prove that their filled manifolds are nondiffeomorphic.
The involution $I$, which inverts all three peripheral directions,
is different from this reflection and preserves each filling.

The four unit-sign cases therefore remain separate candidates unless
an explicit equivalence is proved. Proposition~\ref{prop:variants}
establishes their geometric eligibility, not a reason to expect one
of their groups to be easier. Throughout the rest of the paper the
target remains $V=V_{+1,-1}$.

\partheading{II}{The fundamental group and the remaining filling}

\section{The complete fundamental-group presentation}\label{sec:presentation}

We derive the complement group before imposing either filling slope.
All peripheral loops use the common paths of
Section~\ref{subsec:paths}. The basic transport identity records
where the puncture is based. If a square has lower loop $g$, upper
loop $h$, positive base loop $u$, and puncture meridian $m_0$
based at the lower corner, then
\begin{equation}\label{eq:transport-square}
 gu h^{-1}u^{-1}=m_0,
 \qquad u^{-1}gu=(u^{-1}m_0u)h.
\end{equation}
In particular the correction in the second equation is conjugated.

\subsection{The disk block and the two handles}
Cut the base into a disk and two one-handles, leaving the crossing
of the surgery arcs inside the disk. The fiber collars of $c,e$
are disjoint. Each drilled collar in the disk block is $S^1$
times a three-ball with a straight proper arc removed. The latter
has fundamental group $\Z$; attaching the collar to the remaining
fiber introduces its meridian and commutation with the core.
The disk block therefore has presentation
\begin{equation}\label{eq:disk-block}
 P_0=\langle x,y,r,s,M,N\mid
        [x,y][r,s],\ [M,C],\ [N,E]\rangle.
\end{equation}
Disjointness of the drilled collars supplies no relation $[M,N]$.
Each base one-handle is a product of a single-curve complement
with an interval; its group is generated by the four fiber loops
and the meridian of that curve.

The end inclusions can be read in generators adapted to the
crossings. Put
\[
 f_0=rx,\qquad u=y,\qquad v_1=y^{-1}rs,\qquad n=s^{-1}Ns.
\]
The loops $f_0,r$ are inverse parallels of $c,e$ with paths
missing the opposite core. By \eqref{eq:arcs}, the loop $v_1$
has representative $y_2^{-1}d_c^{-1}s_c$. Its two adjoining
segments meet $c$ on the same side; pushing the half-arc to that
side makes the representative disjoint from $c$. It crosses $e$
once with terminal path $s_3$.
The punctured transport square \eqref{eq:transport-square} now
determines each correction. For $c$ the positive puncture meridian
is $M$, and the crossing contributes its inverse before the path
$u$. Thus $u\mapsto M^{-1}u$, while $f_0,r,v_1,n$ stay fixed.
For $e$ the terminal path is $s_3$, so the same comparison places
$n^{-1}$ after $v_1$: this crossing changes $v_1$ to $v_1n^{-1}$
and fixes $f_0,u,r,n,M$. The side of the correction is part of
the based-path calculation, rather than an interchangeable convention.

Returning to $x=r^{-1}f_0$, $s=r^{-1}uv_1$, and $N=sns^{-1}$
makes the transported meridians explicit. Across $c$ we obtain
\[
 s\longmapsto r^{-1}M^{-1}uv_1=K_0^{-1}s,
 \qquad N\longmapsto K_0^{-1}NK_0,
\]
whereas across $e$ we obtain
\[
 s\longmapsto r^{-1}uv_1n^{-1}=sn^{-1}=N^{-1}s.
\]
The opposite ends are identified by the original monodromies:
$\phi$ exchanges $(x,y)$ with $(r,s)$, and $\psi$ sends
$(x,y,r,s)$ to $(y^{-1},yx,r,s)$. The alpha meridian target is
$K_0$, by the half-arc comparison along $d_c$; the beta target
retains the terminally based meridian $N$. These computations give
all entries of the table below.

\begin{center}
\begin{adjustbox}{max width=\textwidth}
\begin{tabular}{ccccc}
\toprule
Edge generator&Alpha source&Alpha target&Beta source&Beta target\\
\midrule
$x$&$x$&$r$&$x$&$y^{-1}$\\
$y$&$y$&$s$&$M^{-1}y$&$yx$\\
$r$&$r$&$x$&$r$&$r$\\
$s$&$N^{-1}s$&$y$&$K_0^{-1}s$&$s$\\
meridian&$M$&$K_0$&$K_0^{-1}NK_0$&$N$\\
\bottomrule
\end{tabular}
\end{adjustbox}
\end{center}

The beta meridian path crosses $c$ and so acquires the displayed
conjugation by $K_0$. The alpha meridian path misses $e$; its
target path compares along $d_c$, giving $r^{-1}Mr=K_0$.
These paths explain the meridian entries as well as the four
surface entries. Van Kampen imposes $A$(source)$A^{-1}=$(target)
and $B$(source)$B^{-1}=$(target). Injectivity of the edge maps is
not needed for this presentation calculation.

One may check directly that the relations of the edge spaces hold
under both inclusions. For the alpha source,
$[r,N^{-1}s]=[r,s]$ follows from $[N,sr^{-1}s^{-1}]=1$.
For the beta source the transformed surface relation follows from
$[M,rx]=1$. An exact identity making the latter check explicit is
\begin{equation}\label{eq:cross-surface}
 \begin{split}
 [x,M^{-1}y][r,K_0^{-1}s]
   &=(xM^{-1}x^{-1})F_0(xMx^{-1})\\
   &\qquad\cdot(xM^{-1})[M,C]^{-1}(Mx^{-1}),
       \qquad F_0=[x,y][r,s].
 \end{split}
\end{equation}
It is an equality of free words after $K_0=r^{-1}Mr$ and
$C=(rx)^{-1}$ are substituted. Thus there are thirteen complement
relations: the three in \eqref{eq:disk-block} and ten end relations.

\begin{proposition}\label{prop:fifteen}
With the abbreviations \eqref{eq:peripheral}, the group
$G=\pi_1(V)$ has generators $x,y,r,s,A,B,M,N$ and precisely the
following defining relations:
\begin{equation}\label{eq:fifteen}
\begin{aligned}
 \relax[x,y][r,s]&=1,& [M,C]&=1,&[N,E]&=1,\\
 AxA^{-1}&=r,&AyA^{-1}&=s,&ArA^{-1}&=x,\\
 A(N^{-1}s)A^{-1}&=y,\\
 BxB^{-1}&=y^{-1},&B(M^{-1}y)B^{-1}&=yx,\\
 BrB^{-1}&=r,&B(K_0^{-1}s)B^{-1}&=s,\\
 AMA^{-1}&=K_0,&B(K_0^{-1}NK_0)B^{-1}&=N,\\
 M\ell&=1,&ND^{-1}&=1.
\end{aligned}
\end{equation}
\end{proposition}
\begin{proof}
The preceding decomposition proves the first thirteen relations and
their completeness for the complement. Its peripheral triples
commute. Indeed $\ell=Ax=rA$ and $AMA^{-1}=r^{-1}Mr$ give
$\ell M\ell^{-1}=r(AMA^{-1})r^{-1}=M$.
The beta meridian row is $[N,D]=1$. The remaining commutations
$[\ell,C]=1$ and $[D,E]=1$ follow from the two $x,r$ swaps and
the corrected $Bs$ row with $BrB^{-1}=r$, respectively.

The boundary homomorphism $\pi_1(T^3)\to\pi_1(T^2\times D^2)$
is surjective, with kernel generated by the new meridian. Thus
van Kampen for each cap adds exactly its slope relation. By
Definition~\ref{def:V} these are the last two rows of
\eqref{eq:fifteen}. This proves the claimed isomorphism, rather
than only a surjection from a shorter relation group.
\end{proof}

No unpunctured replacement for $As$ or $By$, no relation $[M,N]=1$,
and no relation $M^2=1$ has been used. Adding such a relation would
require a separate proof in the displayed group.

\section{A reversible reduction and a normal generator}\label{sec:reduction}

In the free group on $A,x,y$ define the following words:
\begin{equation}\label{eq:dictionary}
\begin{aligned}
 z&=A^2,&J&=[x,y],&r&=AxA^{-1},&s&=AyA^{-1},\\
 \ell&=Ax,&M&=\ell^{-1},&N&=[s,z^{-1}],&B&=s^{-1}Ns,\\
 K_0&=r^{-1}Mr,&E&=sr^{-1}s^{-1}.&&&&
\end{aligned}
\end{equation}
In particular
$B=z^{-1}s^{-1}zs=A^{-1}y^{-1}A^2yA^{-1}$.
Write
\begin{equation}\label{eq:five}
\begin{aligned}
 \mathcal R&=JAJA^{-1},&\mathcal T&=[A^2,x],\\
 \mathcal X&=BxB^{-1}y,\\
 \mathcal Y&=BAxyB^{-1}x^{-1}y^{-1},\\
 \mathcal S&=BxAsB^{-1}s^{-1}.
\end{aligned}
\end{equation}

\begin{proposition}\label{prop:five}
The dictionary \eqref{eq:dictionary} induces an isomorphism
\[
         G\cong\langle A,x,y\mid
           \mathcal R,\mathcal T,\mathcal X,\mathcal Y,\mathcal S\rangle
\]
that preserves $J=[x,y]$.
\end{proposition}
\begin{proof}
Use $Ax$, $Ay$, and the alpha filling in \eqref{eq:fifteen} to
eliminate $r,s,M$. The corrected $As$ row then solves for
$N=sA^{-1}y^{-1}A=[s,A^{-2}]$. Keep $B$ free for the moment.
Let the beta filling and corrected $Bs$ words be
\[
 F_\beta=NK_0B^{-1},\qquad
 \mathcal S_0=BK_0^{-1}sB^{-1}s^{-1}.
\]
The exact free identity
\begin{equation}\label{eq:beta-elimination}
 s^{-1}F_\beta\mathcal S_0s=(s^{-1}Ns)B^{-1}
\end{equation}
allows replacement of $F_\beta$ by the right side while retaining
$\mathcal S_0$. It is reversible, since
$F_\beta=s(\text{right side})s^{-1}\mathcal S_0^{-1}$.
The new row eliminates $B=s^{-1}Ns$.

The surface, $Ar$, $Bx$, and $By$ rows now become
$\mathcal R,\mathcal T,\mathcal X,\mathcal Y$. The corrected
$Bs$ word satisfies the free identity
\begin{equation}\label{eq:S0}
             \mathcal S_0=(B\mathcal T B^{-1})\mathcal S.
\end{equation}
It may be replaced by $\mathcal S$ because $\mathcal T$ is
retained. The remaining peripheral and meridian rows are products
of conjugates of these relators. Appendix~\ref{app:identities}
prints the exact products, including the conjugators. They prove
both directions of the presentation change and account for every
row of \eqref{eq:fifteen}.
\end{proof}

In particular, the formula for $B$ does not come from deleting a
corrected transport. It uses the beta filling together with the
corrected $Bs$ relation. Likewise \eqref{eq:S0} cannot be replaced
by $\mathcal S_0=\mathcal S$ in a group where $\mathcal T$ is
not known to vanish.

\begin{proposition}\label{prop:J}
The group $G$ is perfect and is normally generated by $J=[x,y]$.
In $G$ one has $AJA^{-1}=J^{-1}$.
\end{proposition}
\begin{proof}
Perfectness is the abelianization calculation in
Section~\ref{sec:topology}. The inversion of $J$ follows from the
surface relation and $r=AxA^{-1}$, $s=AyA^{-1}$.

Work in the quotient obtained by imposing $J=1$. Conjugating
$[x,y]=1$ by $B$ and using the corrected $B$ relations gives
\[
                    [y^{-1},(BMB^{-1})yx]=1.
\]
Since $y$ commutes with $yx$, this forces
$[y,BMB^{-1}]=1$. Conjugating back and using $B^{-1}yB=x^{-1}$
gives $[x,M]=1$. The alpha filling says $M=(Ax)^{-1}$, so
$[x,A]=1$ and $r=x$. As $B$ commutes with $r$, the $Bx$ row now
gives $y=x^{-1}$. Hence $A$ commutes with $y$ too, so $s=y$.
The corrected $As$ row gives $N=1$. The beta filling gives $D=1$,
and the corrected $Bs$ row, written $Ds=sB$, gives $B=1$.
Then $D=BK_0^{-1}=1$ gives $K_0=M=1$. The $By$ row becomes
$y=yx$, so $x=y=1$. Finally $r=s=1$ and $M\ell=1$ give $A=1$.
All generators vanish. Thus $G/\nc J=1$, as asserted.
\end{proof}

\section{An intermediate group containing the surface group}
\label{sec:intermediate}

The successive quotients in the proof of Theorem~\ref{thm:last} are
\[
 \Gamma\twoheadrightarrow\Gamma_1=\Gamma/\nc{\mathcal X}
 \twoheadrightarrow K=\Gamma_1/\nc{\mathcal C}
 \twoheadrightarrow G=K/\nc\lambda.
\]
The word $\mathcal C$ and the attaching word $\lambda$ are defined
in \eqref{eq:CL-words}. The surface group embeds in each of the first
three groups. The reductions retain the
marking of the surface group
\[
                    \Pi=\langle x,y,r,s\mid[x,y][r,s]=1\rangle.
\]
The groups introduced below are intermediate algebraic groups; no
injection into the final filled group is implicit in their definitions.

\subsection{The first of the three remaining relations}
Let
\[
 \Gamma=\langle A,x,y\mid\mathcal R,\mathcal T\rangle,
 \qquad \Gamma_1=\Gamma/\nc{\mathcal X}_\Gamma.
\]
The word $B$ continues to mean the expression in
\eqref{eq:dictionary}, except when it is explicitly introduced as
an independent generator in an equivalent HNN presentation.
Put
\[
                       Q_0=\langle\Pi,B\mid[B,r]=1\rangle.
\]
The substitution $\theta(s)=sB^{-1}$, fixing $x,y,r,B$, is an
automorphism of $Q_0$, with inverse $s\mapsto sB$. Indeed, for
$\varepsilon=\pm1$,
\begin{equation}\label{eq:theta-surface}
 [x,y][r,sB^\varepsilon]
       =[x,y][r,s]\,s[r,B^\varepsilon]s^{-1}.
\end{equation}
Both maps preserve the only relation affected by the substitution.

There is an equivalent presentation
\begin{equation}\label{eq:Gamma-HNN}
 \Gamma=\langle Q_0,A\mid A\theta(g)A^{-1}=\phi(g)
                       \ (g=x,y,r,s)\rangle.
\end{equation}
To verify it algebraically, eliminate $r=AxA^{-1}$ and
$s=AyA^{-1}$. The fourth edge relation gives
$B=A^{-1}y^{-1}A^2yA^{-1}$. The surface and third edge relations
are $\mathcal R,\mathcal T$. The remaining $[B,r]$ row follows
from them by \eqref{eq:NE-product}--\eqref{eq:Br-product} in
the appendix. These eliminations reverse to recover
\eqref{eq:Gamma-HNN}.

The group $Q_0$ is the HNN extension of $\Pi$ centralizing its
infinite cyclic subgroup $\langle r\rangle$, so $\Pi$ embeds
in it. Since $\theta$ is an automorphism and $\phi$ is the
handle-interchange automorphism of $\Pi$, both surface maps in
\eqref{eq:Gamma-HNN} are injective. The HNN normal form theorem
therefore embeds $Q_0$ in $\Gamma$; we use the standard
Bass--Serre description \cite{Serre80}.

The word $\mathcal X=BxB^{-1}y$ is fixed by $\theta$ and its
inverse. Hence
\[
 Q_1=\langle\Pi,B\mid BrB^{-1}=r,\ BxB^{-1}=y^{-1}\rangle,
\]
and
\begin{equation}\label{eq:Gamma1-HNN}
 \Gamma_1=\langle Q_1,A\mid A\theta(g)A^{-1}=\phi(g)
                             \ (g=x,y,r,s)\rangle.
\end{equation}
Both are injective HNN extensions. For the first, the associated
subgroups $\langle r,x\rangle$ and $\langle r,y^{-1}\rangle$
are free on the indicated bases. One direct verification is
\[
 \Pi=\langle F(x,y,r),s\mid srs^{-1}=[x,y]r\rangle,
\]
whose cyclic edge maps inject, so $F(x,y,r)$ embeds by HNN normal
forms. The two rank-two subgroups are then free on their displayed
generators. The second assertion follows because $\theta$ descends
to an automorphism of $Q_1$. Thus $\Pi$ embeds in $\Gamma_1$.

\subsection{Replacing the last two relations}
In words on $x,y,r,s,B$ set
\begin{equation}\label{eq:CL-words}
\begin{aligned}
 d_0&=y^{-1}x^{-1}s,&e_0&=x^{-1}s,&
 \mathcal C&=Bd_0B^{-1}e_0^{-1},\\
 a_*&=B^{-1}yxBy^{-1}x^{-1},&
 m&=yB^{-1}x^{-1}y^{-1}B,&
 \lambda&=Aa_*^{-1}=Axm.
\end{aligned}
\end{equation}
The last equality is free, as is $m=(a_*x)^{-1}$.
The following identities hold in the free group on $A,x,y,s,B$:
\begin{equation}\label{eq:two-rows}
\begin{aligned}
 \mathcal Y&=B\lambda B^{-1},\\
 \mathcal S&=(BxB^{-1})\mathcal Y(BxB^{-1})^{-1}
                 \mathcal X\,x\mathcal Cx^{-1}.
\end{aligned}
\end{equation}
Thus in $\Gamma_1$ each pair of relations implies the other,
and
\begin{equation}\label{eq:final-quotient}
 \nc{\mathcal Y,\mathcal S}_{\Gamma_1}
     =\nc{\mathcal C,\lambda}_{\Gamma_1},\qquad
 K=\Gamma_1/\nc{\mathcal C}_{\Gamma_1},\qquad
 G=K/\nc{\lambda}_K.
\end{equation}
This is an equality of normal closures, not merely an epimorphism
to an easier group.

Both $\theta$ and its inverse fix $\mathcal C$ freely. For example
$d_0\mapsto d_0B^{-1}$ and $e_0\mapsto e_0B^{-1}$, whose added
factors cancel in $Bd_0B^{-1}e_0^{-1}$. Consequently
\begin{equation}\label{eq:Q2}
 Q_2=Q_1/\nc{\mathcal C}_{Q_1}
   =\langle\Pi,B\mid B(r,x,d_0)B^{-1}=(r,y^{-1},e_0)\rangle,
\end{equation}
and $K$ has exactly the HNN presentation in
Theorem~\ref{thm:last}.

\begin{lemma}\label{lem:rankthree}
The subgroups $\langle r,x,d_0\rangle$ and
$\langle r,y^{-1},e_0\rangle$ of $\Pi$ are free of rank three
on the displayed bases.
\end{lemma}
\begin{proof}
Neither subgroup has finite index. An index-$n$ subgroup of a
closed genus-two surface group is the group of a closed surface
of genus $n+1$, whose abelianization needs $2n+2\ge4$ generators.
It cannot be generated by three elements. The covers corresponding
to our subgroups are therefore noncompact surfaces, and their
fundamental groups are free.

In the coordinates $x,y,r,s$ of $H_1(\Pi)=\Z^4$, the two
triples have images
\[
 \begin{array}{lll}
 (0,0,1,0),&(1,0,0,0),&(-1,-1,0,1),\\
 (0,0,1,0),&(0,-1,0,0),&(-1,0,0,1).
 \end{array}
\]
Each triple is linearly independent. The corresponding free group
has rank at least three and is generated by three elements, so its
rank is three. A generating triple of a rank-three free group is
a basis, by Hopficity of finite-rank free groups.
\end{proof}

Lemma~\ref{lem:rankthree} proves that \eqref{eq:Q2} is an
injective HNN extension. Since $\theta$ remains an automorphism,
the subsequent $A$ extension also has injective surface edge maps.
Thus
\begin{equation}\label{eq:surface-injection}
                         \Pi\lhook\joinrel\longrightarrow Q_2
                            \lhook\joinrel\longrightarrow K.
\end{equation}
The nontrivial element $[x,y]$ has infinite order in the surface
group and hence in $K$. Equations \eqref{eq:final-quotient} and
Proposition~\ref{prop:J} now prove Theorem~\ref{thm:last}.

\subsection{The entire kernel of the intermediate quotient}
One additional statement holds for the $\mathcal C$ step, without
restricting the length or form of relations.

\begin{proposition}\label{prop:free-kernel}
The kernel of $\Gamma_1\to K$ is free.
\end{proposition}
\begin{proof}
First consider $Q_1\to Q_2$. Its kernel acts on the $B$-HNN
tree of $Q_1$. Every vertex stabilizer is conjugate to $\Pi$,
which maps injectively into $Q_2$. The kernel therefore acts freely
on that tree and is a free group.

The kernel of $\Gamma_1\to K$ acts on the $A$-HNN tree of
\eqref{eq:Gamma1-HNN}. Its vertex stabilizers are conjugates of
the free kernel just described. Its edge stabilizers are trivial,
because both surface edge groups inject into $K$. The associated
graph of groups has free vertex groups and trivial edge groups,
so its fundamental group is a free product of free groups and a
free group. It is free.
\end{proof}

The standard two-complexes of the presentations of $\Gamma$,
$\Gamma_1$, and $K$ used above are also aspherical. We include the
identification of these particular complexes, since asphericity of
some graph of spaces with the same group would not alone identify a
presentation complex.

First introduce a generator $J$ for $[x,y]$ in $\Gamma$. Before
adjoining $y$, the group
\[
 P_0=\langle A,x,J\mid[A^2,x],\ AJA^{-1}=J^{-1}\rangle
\]
is an injective graph of groups: put $z=A^2$, $v=AJ$, take central
vertex $\langle z,x\mid[z,x]\rangle\cong\Z^2$, and attach two
cyclic vertices $\langle A\rangle,\langle v\rangle$ along
$z=A^2=v^2$. Adjoin $y$ as a cyclic HNN extension with
$yxy^{-1}=J^{-1}x$. Its edge maps inject, since both indicated
words have $x$ exponent one. The tree-of-spaces theorem makes this
complex aspherical \cite[Theorem 1B.11]{Hat02}. Eliminating the
defining generator $v$ using $v=AJ$, then $z$ using $z=A^2$,
and finally $J$ using the $y$-HNN row gives the standard $\Gamma$
complex by elementary homotopy equivalences.

The surface-edge model \eqref{eq:Gamma-HNN} is independently
aspherical by its injective edge maps. The two marked classifying
spaces for $\Gamma$ are homotopy equivalent. Attach the
$\mathcal X$ disk to both: on the graph-of-spaces side this adds
$x$ to the $B$-edge graph, yielding the rank-two edge model of
$\Gamma_1$. The marked equivalence extends across the disk, and
the resulting graph of spaces is aspherical. Repeating with the
$\mathcal C$ disk adds $d_0$ to that edge graph, giving the
rank-three model of $K$. Its edge maps still inject. This proves
the assertion for the actual presentation complexes, without deleting
a redundant two-cell from a higher-dimensional model.

Let $N_C=\ker(\Gamma_1\to K)$. In the universal cover of the
standard $K$ complex, the preimage of the $\Gamma_1$ complex is a
connected $K(N_C,1)$. The relative cellular chain complex consists
of $\Z[K]$ in degree two, from the newly attached disk. Its exact
homology sequence gives
\[
 (N_C)_{\mathrm{ab}}\cong\Z[K]
\]
with the conjugation action. Thus the free kernel in
Proposition~\ref{prop:free-kernel} has countably infinite rank.
The exponent rows of $\mathcal R,\mathcal T,\mathcal X,\mathcal C$
are $(0,0,0),(0,0,0),(0,1,1),(0,0,-1)$, and therefore
\[
 H_1(K)=\Z\langle A\rangle,\qquad H_2(K)=\Z^2,
 \qquad H_i(K)=0\quad(i\ge3).
\]
These conclusions concern the entire intermediate kernel and the
intermediate group. They do not determine $\ker(K\to G)$.

\section{Restoring the meridian transport}\label{sec:transport}

The HNN extension $K$ retains the four surface identifications across
$A$. Restoring the fifth identification, for the meridian, gives a
geometrically familiar intermediate group and a kernel that can be
computed completely. This separates a solved normal-subgroup problem
from the unresolved final filling.

\subsection{A second description of the vertex group}
Put $P=rx$ and $k=r^{-1}mr$, where $m$ is the word in
\eqref{eq:CL-words}. There is a reversible presentation
\begin{equation}\label{eq:Qc}
\begin{aligned}
 Q_c&=\langle\Pi,m\mid[m,P]=1\rangle
      =\Pi*_{\langle P\rangle}(\langle P\rangle\times\langle m\rangle),\\
 Q_2&=\langle Q_c,B\mid
 B(x,m^{-1}y,r,k^{-1}s)B^{-1}=(y^{-1},yx,r,s)\rangle.
\end{aligned}
\end{equation}
Here $m$ is an independent generator in the second presentation;
eliminating it recovers its earlier word. Thus the notation does not
impose the final relation $m=(Ax)^{-1}$ prematurely.

We verify the presentation change. The corrected $By$ relation solves
for $m=yB^{-1}x^{-1}y^{-1}B$. Put $V_0=sB^{-1}s^{-1}$.
The surface relation and $[B,r]=1$ give
$rV_0r^{-1}=J^{-1}V_0J$. The equation $\mathcal C=1$ says
$V_0=xyB^{-1}x^{-1}$. Using $BxB^{-1}=y^{-1}$ gives
\[
 J^{-1}V_0J=yB^{-1}x^{-1}y^{-1}=mB^{-1}.
\]
Consequently $\mathcal C=1$ is equivalent to
$m=rsB^{-1}s^{-1}Br^{-1}$, which is the corrected $Bs$ relation.
All substitutions reverse. The remaining peripheral relation follows
from the free identity, with $F_0=[x,y][r,s]$,
\begin{equation}\label{eq:crossing-identity}
 [x,m^{-1}y][r,k^{-1}s]
  =wF_0w^{-1}v[m,P^{-1}]^{-1}v^{-1},
 \qquad w=xm^{-1}x^{-1},\quad v=xm^{-1}.
\end{equation}
The four beta equations conjugate its left side by $B$ to $F_0$;
hence $[m,P]=1$. Conversely, the equations in \eqref{eq:Qc}
imply the original three equations for $Q_2$.

The crossing map
\[
 \theta_c(x,y,r,s,m)=(x,m^{-1}y,r,k^{-1}s,m)
\]
is an automorphism of $Q_c$, with inverse
$(x,my,r,ks,m)$. Identity \eqref{eq:crossing-identity} verifies
preservation of the surface relation; for the inverse the free identity is
\[
 [x,my][r,ks]=(xmx^{-1})F_0(xm^{-1}x^{-1})
                   \cdot x[m,P^{-1}]x^{-1}.
\]
Both maps fix the peripheral relation.
The target map $(x,y,r,s)\mapsto(y^{-1},yx,r,s)$ is an automorphism
of $\Pi$. Thus \eqref{eq:Qc} is also an injective HNN extension,
and $Q_c$ embeds in $Q_2$.
The map $\phi$ extends to an involutive automorphism of $Q_c$ by
$\phi(m)=k$: it carries $[m,P^{-1}]$ to its conjugate by $r^{-1}$,
and $\phi^2(m)=P^{-1}mP=m$. The automorphism $\theta$ of $Q_2$
fixes $m$, since its defining word contains no $s$.

\subsection{A free kernel with an explicit module}
Write
\[
 U=Ax,\qquad \delta=[U,m],\qquad
 H_\beta=K/\nc\delta,\qquad
 \mathcal D=\ker(K\to H_\beta),\quad
 \mathcal N=\ker(K\to G).
\]
We use $H_\beta$, rather than $H$, for this group to distinguish it
from the doubled section. For a subgroup $L$ of a group $H_\beta$,
$\Z[H_\beta/L]$ denotes the permutation module on its left cosets.

\begin{proposition}[The meridian-transport kernel]\label{prop:transport}
The group $H_\beta$ is the fundamental group after the beta filling
alone. It has the injective HNN presentation
\[
 H_\beta=\langle Q_2,A\mid
 A\theta(g)A^{-1}=\phi(g)\quad(g=x,y,r,s,m)\rangle.
\]
The kernel $\mathcal D$ is free of countably infinite rank, and
\[
 \mathcal D_{\mathrm{ab}}\cong\Z[H_\beta/\langle P\rangle]
\]
as a module under conjugation. Furthermore
\[
 \mathcal D\subset\mathcal N,\qquad
 \mathcal N/\mathcal D\cong\ker(H_\beta\to G),\qquad
 G=H_\beta/\nc\lambda.
\]
\end{proposition}
\begin{proof}
Since $rA=Ax=U$,
\[
 r(AmA^{-1}k^{-1})r^{-1}=[U,m].
\]
Thus the new relation is precisely the missing meridian transport.
Both enlarged edge maps are injective by \eqref{eq:Qc}. To check
the geometric description, perform only the beta filling in the
fifteen-row presentation and eliminate $N=D$. The beta equations
give $D=sBs^{-1}$ and $[D,E]=1$; the alpha source $N^{-1}s$
becomes $sB^{-1}$. The surviving equations are exactly the displayed
HNN presentation. Also $\delta=[Um,m]=[\lambda,m]$, so
$\delta$ lies in the normal closure of $\lambda$ in $K$.

Let $T_K$ be the $A$-HNN tree of $K$. Since $Q_2$ injects into
$H_\beta$, the kernel $\mathcal D$ acts freely on $T_K$ and is
free. Put $E_0=\theta(\Pi)$ and $E_1=\theta(Q_c)$. The graph
$\mathcal D\backslash T_K$ has vertices $H_\beta/Q_2$ and
oriented edges $H_\beta/E_0$. The tree of $H_\beta$ has the same
vertices and edges $H_\beta/E_1$. The natural map between these
graphs identifies parallel edges and is the identity on vertices.
Since the boundary map of a tree is injective,
\[
 \mathcal D_{\mathrm{ab}}
 =\ker\bigl(\Z[H_\beta/E_0]\longrightarrow
                    \Z[H_\beta/E_1]\bigr).
\]
Now $E_1=\langle E_0,m\mid[m,P]=1\rangle$. Its own HNN tree
gives the exact augmented chain sequence
\[
 0\longrightarrow\Z[E_1/\langle P\rangle]
 \longrightarrow\Z[E_1/E_0]\longrightarrow\Z\longrightarrow0.
\]
Inducing this sequence to $H_\beta$ proves the asserted module
isomorphism. Induction is exact because $\Z[H_\beta]$ is free
as a right $\Z[E_1]$-module. Finally the height map, with $A$
of height one and $Q_2$ of height zero, shows that
$H_\beta/\langle P\rangle$ is countably infinite.
\end{proof}

This result determines an entire kernel, rather than a collection of
short relations. It does not determine the larger kernel $\mathcal N$.
The distinction can be seen in an explicit subgroup. Amalgam normal
forms in $Q_c$, followed by its inclusion in $Q_2$, give
\[
 \langle P,m\rangle\cap\Pi=\langle P\rangle,
 \qquad \langle P,m\rangle\cap\theta(\Pi)=\langle P\rangle.
\]
The second equality follows by applying the automorphism $\theta$
of $Q_2$, which fixes $P$ and $m$; it does not require
$\theta(\Pi)$ to lie in $Q_c$. Replacing the
stable letter $A$ by $U=Ax$, its edge map fixes $P$. Britton's lemma
therefore gives
\begin{equation}\label{eq:local-product}
 \langle P,m,U\mid[P,m]=[P,U]=1\rangle
       \cong\Z\times F(m,U)\ \lhook\joinrel\longrightarrow K.
\end{equation}
In particular $\delta\ne1$ in $K$. In $H_\beta$ the same three
generators give $\Z^3$: $P,m$ generate the embedded peripheral
$\Z^2$, and height separates $U$ from it.

These subgroups give necessary conditions for proposed free-kernel
arguments. If $\mathcal N$ were free, $P$ would have infinite order
in $G$; otherwise $\langle\lambda,P^d\rangle\cong\Z^2$
would lie in $\mathcal N$ for some $d>0$. Similarly,
\begin{equation}\label{eq:kernel-necessary}
 \ker(H_\beta\to G)\text{ free}\quad\Longrightarrow\quad
                  \langle P,m\rangle\lhook\joinrel\longrightarrow G.
\end{equation}
Indeed a nontrivial element $w$ in the kernel on that torus would give
$\langle\lambda,w\rangle\cong\Z^2$ in the larger kernel.
Neither freeness assertion is established.

\section{Restrictions on possible filling arguments}\label{sec:limits}

The final attaching word has height one, but it is a stable letter
times a vertex word. The following restrictions explain why several
familiar injection theorems do not immediately settle its effect.
They concern the displayed intermediate groups; they are not evidence
for triviality of $G$.

\begin{proposition}\label{prop:tree-limits}
An almost-malnormal subgroup of $K$ containing a conjugate of a
nonzero power of $\lambda$ or $P$ is all of $K$. Moreover every
action of $K$ or $H_\beta$ without inversions on a simplicial tree
with finite edge stabilizers has a global fixed vertex.
\end{proposition}
\begin{proof}
Recall that $L\le K$ is almost malnormal if $L\cap gLg^{-1}$
is finite whenever $g\notin L$. Such a subgroup contains the full
centralizer of each of its infinite-order elements. In $K$ we have
\begin{equation}\label{eq:centralizer-chain}
 [U,P]=[m,P]=[\lambda,P]=1,\quad U^2=A^2P,\quad
 [B,r]=1,\quad x=A^{-1}U,\quad y=Bx^{-1}B^{-1}.
\end{equation}
The surface injection makes $P,r$ infinite, and height makes
$A^2,\lambda$ infinite. Starting from $\lambda^n\in L$ puts
$\lambda,P,U,A^2,A,x,r,B,y$ in $L$ in that order. Starting from
$P^n$ omits only the first step. Conjugating $L$ handles the stated
conjugates.

For the tree assertion, $\langle P,U\rangle\cong\Z^2$ in either
group. If $P$ were hyperbolic, its centralizer would preserve its axis
and act by translations: a reflection would conjugate its nonzero
translation to its inverse. The homomorphism $\Z^2\to\Z$ would have infinite
kernel fixing an edge, contrary to finite edge stabilizers. Hence
$P$ fixes a vertex $v$, uniquely because it has infinite order.
Its centralizer fixes $v$, so $U$ does. Then $A^2=U^2P^{-1}$
fixes $v$; this infinite-order element also has a unique fixed vertex,
so $A$ fixes $v$. Now $x,r$ fix $v$, and the centralizer of the
infinite-order element $r$ contains $B$. Finally $y$ fixes $v$.
Thus the generators fix $v$.
\end{proof}

Hyperbolically embedded subgroups are almost malnormal
\cite[Proposition 4.33]{DGO17}. Consequently the final slope cannot
be placed in a proper such subgroup of $K$ to invoke a deep-filling
theorem. Taking a nonzero power of the slope does not remove this
obstruction. The tree assertion likewise rules out applying a theorem
that requires a tree with trivial edge stabilizers on which the
attaching word is hyperbolic; compare \cite[Theorem 5.11]{ADL11}.
It does not conflict with our HNN descriptions, whose edge groups
are infinite.

There are two further, distinct issues with the displayed HNN quotient.
First, $\lambda=A(xm)$ is cyclically conjugate to $(xm)A$.
It is precisely the stable-letter-length-one case excluded in
\cite[Theorem 6.1]{Chen25}. Second, the surface edge group is not
malnormal in $Q_2$: $m\notin\Pi$ but $m$ centralizes the nontrivial
$P\in\Pi$. Applying $\theta$ gives the same obstruction for the
other surface edge image. A theorem for an equation over the free
product $Q_2*\langle A\rangle$ cannot simply retain the HNN
identifications without verifying its hypotheses for them. Eliminating
$A$ in the actual quotient imposes those further equations on $Q_2$.

\subsection{Two obstructions to aspherical decompositions}\label{subsec:cuts}

Aspherical-gluing results for other Luttinger-surgery constructions
suggest a geometric approach, but require injective gluing faces and
the appropriate boundary path conditions. Baykur's recent treatment
\cite[Theorem 3, Lemma 6, and Remark 15]{Bay26} does not include this
construction. Two natural choices of faces can be excluded here.

First, there is no monodromy-invariant multicurve containing the alpha
curve $c$, even after any finite cover of the base. Direct substitution
in the marked beta monodromy gives
\[
 \psi(J)=J,\quad \psi^6(x)=J^{-1}xJ,\quad
 \psi^6(y)=J^{-1}yJ,\qquad
 \psi^{6n}(C)=J^{-n}x^{-1}J^nr^{-1}.
\]
Split the surface group as
$F(x,y)*_{\langle J\rangle}F(r,s)$, with
$J=[r,s]^{-1}$ in the second factor. The edge subgroup is maximal
cyclic and malnormal in each free factor. The words displayed for
$\psi^{6n}(C)$ are cyclically reduced of amalgam length two.
If the word for $n$ were conjugate to the word for zero, the conjugacy
normal form would give integers $k_0,k_1$ with
\[
 J^{-n}x^{-1}=J^{k_0}x^{-1}J^{-k_1},\qquad
 J^nr^{-1}=J^{k_1}r^{-1}J^{-k_0}.
\]
Malnormality first gives $k_1=0$, $k_0=-n$, and then $n=0$.
Their homology classes also exclude conjugacy to an inverse.
Applying a power of $\psi$ proves pairwise distinction. Every
finite-index base subgroup contains a positive power of the beta
loop, so no finite invariant set containing $c$ appears in a finite
base cover. This excludes a global fiberwise cut, not cuts over
proper subsets of the base or more general hypersurfaces.

Second, consider the exterior $S_c$ of $c$ in the closed fiber mapping
torus with monodromy $\psi$. Its group is precisely the $Q_2$ presentation in
\eqref{eq:Qc}, before the $A$ attachment, and
\[
 H_1(S_c)=\Z\langle B,r,s\rangle,
 \qquad x=y=m=0,\qquad [C]=-[r].
\]
Its boundary torus is generated by $C,m$. After the alpha filling,
an annulus in its filling solid torus joins the old meridian to the
new core: in longitude and normal-disk coordinates it is
$(v,\rho)\mapsto(-v,\rho e^{2\pi iv})$. Taking its product with
$C$ extends $S_c$ to a diffeomorphic transverse piece $S_1$ with
boundary on the alpha core torus. An identity-side push-off survives
the beta surgery and has the same based generator images.

No compact completion of this entire $S_1$ across its boundary torus
can map injectively on fundamental groups to $V$ while extending that
inclusion. Indeed Mayer--Vietoris for $S_1\cup_T L$ shows that the
kernel of $H_1(S_1)\to H_1(S_1\cup_T L)$ lies in
$\operatorname{im}H_1(T)=\Z\langle r\rangle$. Thus the word
\[
                    \kappa_y=m^{-1}x^{-1}yxms^{-1}
\]
has infinite-order homology class $-[s]$ in every such completion.
Its image in $V$ is trivial, since the alpha filling gives
$A=m^{-1}x^{-1}$ and the retained relation is $AyA^{-1}=s$.
The same argument applies before the beta filling, and also to
$S_1$ without a cap. This excludes arbitrary compact caps along the
whole boundary torus, not modifications of the transverse piece itself.
In particular these two natural decompositions do not supply the
injective faces required for an asphericity proof.

\subsection{Normal generation is sensitive to the attaching word}
Although $A$ normally generates $K$, this does not prove that
$\lambda$ does. Indeed, setting $A=1$ in $K$ gives $s=y$, $B=1$,
$\mathcal X=xy$ and $\mathcal C=y^{-1}$, so the quotient is trivial.
There is, however, no endomorphism of $K$ that simultaneously sends
$A$ to $g\lambda^{\varepsilon}g^{-1}$ and $J$ to
$gJ^{\epsilon}g^{-1}$ for a common $g\in K$ and arbitrary signs
$\varepsilon,\epsilon\in\{1,-1\}$.

To see the obstruction, the surface relation requires $A$ to invert
$J$. Since $A$ already does so in $K$, replacing it by
$\lambda=A(xm)$ would require $[xm,J]=1$; inversion of either
$A$ or $J$ gives the same requirement. But
\[
 [J^{-1},xm]=(J^{-1}x)mJm^{-1}x^{-1}
\]
is a reduced alternating word in
$\Pi*_{\langle P\rangle}(\langle P\rangle\times\langle m\rangle)$.
Its three surface syllables lie outside $\langle P\rangle$:
for the outer two use abelianization, and for $J$ use its nontriviality
and its zero abelianization. Thus the word is nonidentity in $Q_c$
and $K$. This excludes the indicated substitution, not every possible
endomorphism and not normal generation by $\lambda$ itself.

\subsection{Collapse criteria and two-dimensional representations}
\begin{proposition}\label{prop:collapse}
Any quotient of $G$ satisfying one of
\[
 [A^2,y]=1,\qquad[A,r]=1,\qquad[x,r]=1,\qquad M^2=1
\]
is trivial. For every field $k$, every homomorphism from $G$ to
$\operatorname{PGL}_2(k)$ or $\operatorname{GL}_2(k)$ is trivial.
\end{proposition}
\begin{proof}
If $[A^2,y]=1$, the corrected $As$ row gives $N=1$.
The beta filling then gives $D=1$. The corrected $Bs$ row says
$Ds=sB$, so $B=1$ and then $K_0=M=1$. The $Bx,By$ rows give
$x=y=1$; the remaining rows give $r=s=A=1$.
If $[A,r]=1$, then $x=r$ and $B$ commutes with $x$, so
$y=x^{-1}$ and the first case applies.
If $[x,r]=1$, conjugation by $B$ gives $[y,r]=1$. The formula
\[
 M=B^{-1}\bigl((ByB^{-1})(yx)^{-1}\bigr)B
\]
then gives $[M,r]=1$. Since $M=(Ax)^{-1}$, this implies $[A,r]=1$.
Finally $M^2=1$ implies $(Ax)^2=1$, whence
$A^2=(rx)^{-1}$. Conjugating by $A$ gives also $A^2=(xr)^{-1}$,
and the third case applies.

For a projective representation choose matrix lifts $U_0,T_0$ of
$A,r$. The relation $[A^2,r]=1$ gives
$T_0U_0^2T_0^{-1}=cU_0^2$, where determinants imply $c^2=1$.
If $U_0^2$ is scalar, the first collapse criterion applies. If it is
nonscalar and $c=1$, its matrix centralizer is the commutative algebra
of polynomials in it; hence $U_0,T_0$ commute, and the second criterion
applies. Otherwise $c=-1\ne1$. Over an algebraic closure $U_0^2$
has distinct nonzero eigenvalues $d,-d$, since its trace is zero.
In that eigenbasis $U_0$ is diagonal and $T_0$ is off-diagonal.
Thus $(T_0U_0)^2$ is scalar. Since $rA=Ax=M^{-1}$, the fourth
criterion applies. The argument includes characteristic two in the
$c=1$ case. Finally a linear representation has scalar image after
projectivization, and that image is trivial because $G$ is perfect.
\end{proof}

Thus two-dimensional representations cannot detect nontriviality of
$G$. Likewise a
bi-orderable image is trivial: conjugation preserves its order but
$A$ conjugates $J$ to $J^{-1}$, forcing $J=1$ and then the entire
image to vanish.

\section{The universal central extension}\label{sec:central}

There is a second precise algebraic question, sufficient for
nontriviality but not equivalent to it. It concerns a cyclic central
kernel, whose geometric interpretations will occupy Part III.
Use the five words $\mathcal R,\mathcal T,\mathcal X,\mathcal Y,
\mathcal S$ of \eqref{eq:five} and define
\begin{equation}\label{eq:universal-central}
 \widehat G=\langle A,x,y,t\mid
 \mathcal Rt^{-1},\mathcal T,\mathcal X,\mathcal Y,\mathcal S,
 [t,A],[t,x],[t,y]\rangle.
\end{equation}
A central extension is universal if it admits a unique homomorphism
over $G$ to every central extension of $G$.

\begin{theorem}\label{thm:central}
The map $\widehat G\to G$ fixing $A,x,y$ and killing $t$ is the
universal central extension. Its kernel is the cyclic central subgroup
$\langle t\rangle$, naturally isomorphic to $H_2(G;\Z)$.
Under the dictionary \eqref{eq:dictionary}, the surface row of the
original presentation has value $t$ and the other fourteen rows have
value one. No order of $t$ is asserted.
\end{theorem}
\begin{proof}
The exponent matrix of $\mathcal X,\mathcal Y,\mathcal S$, with
columns $A,x,y$, is
\[
 C_0=\begin{pmatrix}0&1&1\\1&0&0\\1&1&0\end{pmatrix},
 \qquad \det C_0=1.
\]
Take any central extension $C\to D\to G$ and arbitrary lifts of
the three generators. Write $\eta_X,\eta_Y,\eta_S\in C$ for
the three corresponding row values. Multiply the lifts of $A,x,y$,
respectively, by
\[
 \eta_Y^{-1},\qquad \eta_Y\eta_S^{-1},\qquad
 \eta_X^{-1}\eta_Y^{-1}\eta_S.
\]
The exponent matrix shows that the three new row values are one.
The values of $\mathcal R,\mathcal T$ are unchanged, since these
words have zero exponent vectors.

Appendix~\ref{app:central-identity} prints a free identity consisting
of nineteen conjugates of the five relators, whose signed multiplicities
in their displayed order are $(0,1,0,0,0)$. Evaluate this identity
on the normalized lifts. Every row value is central; the identity
therefore forces the value of $\mathcal T$ to be one. Assigning $t$
the value of $\mathcal R$ gives a homomorphism $\widehat G\to D$
over $G$.

The kernel of $\widehat G\to G$ is exactly $\langle t\rangle$,
since killing $t$ gives the presentation of $G$. Also $\widehat G$
is perfect: the three exponent rows kill $A,x,y$ in its abelianization,
and $\mathcal Rt^{-1}$ then kills $t$. Two homomorphisms to $D$
over $G$ differ by a homomorphism from $\widehat G$ to the abelian
central group $C$. Perfectness makes it trivial, proving uniqueness.
The identification of the kernel with $H_2(G;\Z)$ is the standard
kernel theorem for universal central extensions, or equivalently the
Hopf formula. Finally, the products in Appendix~\ref{app:identities}
express every original row in the reduced rows. Apart from the surface
row, every product has total $\mathcal R$ exponent zero. Its central
factors thus cancel, proving the last assertion.
\end{proof}

If $t\ne1$, then $G\ne1$: otherwise $\widehat G$ would be both
cyclic and perfect. The converse is not justified. A nontrivial perfect
group can have zero second integral homology. The central question must
therefore be kept distinct from Question~\ref{q:decisive}.

One immediate limitation concerns the standard presentation complex
of \eqref{eq:five}. It has $H_2\cong\Z^2$, by the displayed exponent
matrix and its two zero rows. It cannot be aspherical, since its
fundamental group has cyclic $H_2$. Thus a proof requiring this
particular two-complex to be aspherical cannot succeed. This observation
does not exclude other classifying spaces for $G$ and does not determine
which additional spherical relations may kill $t$.

\partheading{III}{Geometric interpretations of the central kernel}

\section{Two closed forms and a primitive circle bundle}\label{sec:circle}

The central generator in \eqref{eq:universal-central} can be realized
as a circle fiber. This realization translates its possible order into
an exact question about spherical classes in $V$.

\begin{proposition}\label{prop:pair}
The specified filling $V$ admits closed two-forms $\omega_1,\omega_2$
of constant rank two, with
\[
 \omega_1^2=\omega_2^2=0,\qquad \omega_1\wedge\omega_2>0,
 \qquad \int_F\omega_1=1,\qquad [\omega_2]=0.
\]
The class $e=[\omega_1]$ is the real image of the integral generator
of $H^2(V;\Z)$ evaluating as one on $F$.
\end{proposition}
\begin{proof}
On the original flat surface bundle choose the fiber area form
invariant under the half-turn and preserved by the annular twists.
It descends to a closed rank-two form $\Omega_F$. A positive base
area form gives $\Omega_B$, with complementary kernel. In the
adopted alpha chart these forms are
\begin{equation}\label{eq:pair-chart}
 \Omega_F=(dQ_1-\tfrac12dQ_2)\wedge dY_1,
 \qquad
 \Omega_B=dQ_2\wedge(dY_2+\tfrac12dY_1).
\end{equation}
Their sum is the cotangent symplectic form. At beta the corresponding
forms are $dQ_1\wedge dY_1$ and $dQ_2\wedge dY_2$.

Use the surgery shear of \cite[Section 2.1]{ADK03} in the coordinates
$(x_1,x_2,y_1,y_2)=(Q_2,Q_1,Y_2,Y_1)$. On the upper side of a
normal-square annulus it replaces $Q_2$ by $Q_2+\chi(Y_2)$,
where $\chi$ is a step function with derivative supported in the
middle third. Its derivative replaces $dQ_2$ by
$dQ_2+a(Y)dY_2$ on that annulus. The pullbacks in
\eqref{eq:pair-chart} are respectively
\[
 \Omega_F+\tfrac12a(Y)dY_1\wedge dY_2,
 \qquad \Omega_B-\tfrac12a(Y)dY_1\wedge dY_2.
\]
Take the normal square $[-r_0,r_0]^2$ with inner square
$[-r_0/2,r_0/2]^2$, and choose a smooth function $\rho(Y_1)$
equal to zero for $Y_1\le0$ and one for $Y_1\ge r_0/2$.
The support of $\chi'$ lies in $|Y_2|<r_0/3$. The function
$f(Y)=\rho(Y_1)\chi'(Y_2)/2$ agrees with $a/2$ throughout the
gluing annulus: on its upper and lower strips $\rho$ has the
required constant value, and on the lateral strips $\chi'=0$.
The two forms
\[
 \widetilde\Omega_F=\Omega_F+f(Y)dY_1\wedge dY_2,
 \qquad
 \widetilde\Omega_B=\Omega_B-f(Y)dY_1\wedge dY_2
\]
therefore extend across the inserted piece. They are closed, their
squares vanish, and their product equals $\Omega_F\wedge\Omega_B$:
the added normal two-form wedges to zero with each original form.
At beta the shear preserves both forms separately. These constructions
use precisely the existing two gluing maps and filling slopes.

Finally divide the extended fiber form by its positive integral over
$F$. Since $H_2(V;\Z)=\Z[F]$, the resulting class is the stated
integral generator. The base form integrates to zero over $F$ and
is consequently exact.
\end{proof}

Such a pair is called a \emph{symplectic pair}: the kernels are
complementary foliations, by closedness and constant rank
\cite{BK06}. This terminology asserts more than the existence of
the symplectic sum, but does not assert that either foliation has
simply connected leaves.

\begin{theorem}\label{thm:circle}
Let $p:\mathcal P\to V$ be the principal circle bundle of Euler
class $e$, with five-dimensional total space. Then $\mathcal P$ has the integral homology of $S^3$,
and $\pi_1(\mathcal P)\to G$ is the universal central extension.
The oriented circle fiber corresponds to $t$ up to inversion.
A connection $\alpha$ can be chosen so that, with
$\beta=p^*\omega_2$,
\[
 d\alpha=p^*\omega_1,\qquad
 \alpha\wedge d\alpha\wedge\beta\ne0,
 \qquad (d\alpha)^2=\beta^2=0,
\]
and $\beta$ is exact.
\end{theorem}
\begin{proof}
The integral homology spectral sequence of this principal bundle has
four nonzero entries, at $(0,0),(0,1),(2,0),(2,1)$, all equal to
$\Z$. Its only nonzero differential is cap product with $e$,
\[
 d_2:H_2(V;\Z)\longrightarrow H_0(V;\Z),\qquad [F]\longmapsto1.
\]
It is an isomorphism, leaving precisely $H_0(\mathcal P)=
H_3(\mathcal P)=\Z$. Write $\Pi_{\mathcal P}=\pi_1(\mathcal P)$.
The Hopf homology sequence gives $H_1(\Pi_{\mathcal P})=
H_2(\Pi_{\mathcal P})=0$. The homotopy sequence gives a central
extension of $G$ by the image of the circle group. Pulling an arbitrary
central extension with kernel $C$ back to $\Pi_{\mathcal P}$ gives
an extension classified by $H^2(\Pi_{\mathcal P};C)=0$.
It splits uniquely, since $\operatorname{Hom}(\Pi_{\mathcal P},C)=0$.
This proves the universal property and identifies the extension with
Theorem~\ref{thm:central}.

Over the genus-two fiber $F$ the circle bundle has Euler number one.
The product of the two commutators of lifted surface generators is
therefore the circle fiber or its inverse. Central changes of the
lifts do not change the commutators. The same surface word has value
$t$ in \eqref{eq:universal-central}, proving the marking assertion.
A connection with the prescribed curvature exists because $\omega_1$
represents the Euler class. The remaining identities follow from
Proposition~\ref{prop:pair} and the connection splitting.
\end{proof}

The pair $(\alpha,\beta)$ is a contact-symplectic pair in the
terminology of \cite{BK06}. Its existence does not imply that the
circle fiber injects on fundamental groups. The exact obstruction is
visible in the homotopy sequence:
\begin{equation}\label{eq:spherical-periods}
 \pi_2(V)\xrightarrow{\langle e,-\rangle}\Z
 \longrightarrow\widehat G\longrightarrow G\longrightarrow1.
\end{equation}
If the image of the first map is $d\Z$, with $d\ge0$, then
\begin{equation}\label{eq:order-t}
                 \langle t\rangle\cong\Z/d\Z,
\end{equation}
where $d=0$ means infinite cyclic. Thus a restriction $d\ne1$
would prove $G\ne1$. No such restriction has been established.
In particular the minimum genus bounds for embedded surfaces do not
apply to an arbitrary mapped sphere representing a class in $\pi_2(V)$.

For comparison, the product pair on $S^2\times D^2$ has primitive
circle bundle $S^3\times D^2$. It satisfies the same rank and volume
identities, yet its Hopf circle is nullhomotopic. This example does
not have the genus-two geometry of our construction. It shows exactly
why the pair and the homology calculation alone cannot prove survival.

\section{Removing the section}\label{sec:exterior}

Let $\Sigma_O$ be the proper punctured-torus section through the
center $O$ of the marked octagon, and set
\[
 X=V\setminus\operatorname{int}_V\nu\Sigma_O,\qquad
 Q_X=\pi_1(X).
\]
Choose the normal disk bundle small enough to miss the two surgery
neighborhoods and round the resulting corners. The positive normal
meridian, based radially from the initial octagon vertex, is denoted
$\mu$.

\begin{proposition}\label{prop:exterior}
The group $Q_X$ has the fourteen relations of
Section~\ref{sec:presentation} other than the surface relation, with
\[
 \mu=[x,y][r,s],\qquad G=Q_X/\nc\mu.
\]
Its integral homology is
\[
 H_1(X)=H_3(X)=0,\qquad H_2(X)=\Z^2,
\]
where the two generators are normal circle bundles over a basis of
$H_1(\Sigma_O)$. Both $\omega_1|_X$ and $\omega_2|_X$ are exact.
If
\[
 C_\mu=\nc{[\mu,q]:q\in Q_X}_{Q_X},
\]
then $Q_X/C_\mu\cong\widehat G$ over $G$, with $\mu$ corresponding
to $t$ up to the common orientation convention.
\end{proposition}
\begin{proof}
The geometric calculation must be repeated in the punctured fiber:
simply deleting a relator after a closed-surface simplification would
not justify the assertion. Remove a small disk about $O$. The remaining
polygon retracts radially onto its outer boundary, and the punctured
fiber retracts onto the rose with edges $x,y,r,s$. Its positive
normal meridian is $[x,y][r,s]$; the oriented boundary of the punctured
fiber has the opposite orientation.

Here are the peripheral checks in this free group. Retain the edge
parameters of Section~\ref{sec:construction}. The first alpha half-arc
runs from $E_6(7/20)$ to $E_4(13/20)$ and returns to the vertex
along $E_4$. With its initial edge segment it lies in a sector of
angular width less than $\pi$ and retracts onto $E_5$ backwards,
giving $r^{-1}$. For the second half-arc use the paired segment on
$E_8$ from parameter $1$ to $13/20$, then the chord to $E_2(7/20)$
and the returning edge segment. Its sector contains $E_8,E_1,E_2$
and gives $x$. Both homotopies avoid $O$. Hence the alpha longitude
is $C=x^{-1}r^{-1}$ and its section is still $\ell=Ax$.
The beta chord from $E_6(13/20)$ to $E_8(7/20)$, with its based
segments, retracts onto $E_6E_7E_8$, giving $E=sr^{-1}s^{-1}$.
Its transport rectangle still crosses $c$ once, so $D=BK_0^{-1}$.

The disk block now has the free fiber group and the two relations
$[M,C]=[N,E]=1$. All ten end inclusions in the table of
Section~\ref{sec:presentation} remain valid: their paths and
rectangles lie in these sectors and the two torus collars, disjoint
from $O$. In particular the beta meridian source is still
$K_0^{-1}NK_0$, not $N$. The half-turn sends the four free generators
to $r,s,x,y$, and the twists give $(y^{-1},yx,r,s)$ by homotopies
supported away from $O$. Neither uses the surface relation.
The two filling slopes are unchanged. Van Kampen therefore gives
exactly the asserted fourteen rows. Reattaching
$\Sigma_O\times D^2$ kills precisely its normal meridian.

Excision identifies $H_j(V,X)$ with $H_{j-2}(\Sigma_O)$.
The map $H_2(V)\to H_2(V,X)$ is intersection with the section,
hence an isomorphism. The long exact sequence gives the asserted
homology and identifies $H_2(X)$ with $H_1(\Sigma_O)$ by the
normal-torus boundary map. It also makes $H_2(X)\to H_2(V)$ zero.
Consequently restriction kills $H^2(V;\mathbb R)$, proving exactness
of both forms.

The Euler class $e$ is dual to the proper section. Its line bundle
has a section vanishing simply on $\Sigma_O$, so the principal
circle bundle is trivial over $X$. It is also trivial over
$\Sigma_O\times D^2$. The transition function has degree one on
each normal circle. In the $X$ trivialization, write $h$ for the
principal circle. The common boundary maps onto the fundamental
group of the section neighborhood times that circle, and its normal
kernel generator maps to $\mu h^{-1}$. Thus
\[
 \pi_1(\mathcal P)
 =(Q_X\times\langle h\rangle)/\nc{\mu h^{-1}}
 =Q_X/C_\mu.
\]
Theorem~\ref{thm:circle} identifies this group with $\widehat G$.
\end{proof}

There is a useful algebraic check on the missing surface row. The
same defining eliminations as in Proposition~\ref{prop:five} give
\begin{equation}\label{eq:exterior-five}
 Q_X=\langle A,x,y\mid
 \mathcal T,\mathcal X,\mathcal Y,\mathcal S,[A^2,J]\rangle.
\end{equation}
The last relator is needed in this derivation. With $z=A^2$ and $u_*,v_*,w_*$
as in Appendix~\ref{app:identities}, the exact comparison is
\[
 [N,E]=(u_*\mathcal T^{-1}u_*^{-1})
        (u_*[z,J]u_*^{-1})(v_*\mathcal T v_*^{-1}),
 \qquad [B,r]=w_*[N,E]^{-1}w_*^{-1}.
\]
All other eliminations use the retained rows. Only after the surface
word $\mathcal R=JAJA^{-1}$ is made central does
$AJA^{-1}=J^{-1}\mathcal R$ imply $[A^2,J]=1$ automatically.
Adjoining $t=\mathcal R$ then recovers \eqref{eq:universal-central}.
This verifies the same quotient without assuming that a peripheral
group survives centralization.

The geometric central-kernel question is therefore exactly
\begin{equation}\label{eq:meridian-question}
                         \mu\notin C_\mu.
\end{equation}
Even proving that $\mu$ has infinite order in $Q_X$ would not by
itself establish \eqref{eq:meridian-question}.

\section{A Liouville domain with hypertight boundary}\label{sec:contact}

The exterior has additional symplectic geometry. We give the boundary
calculation because it distinguishes noncontractibility in the boundary
from the stronger assertion required in the filling or its central
quotient. A Liouville domain is a compact exact symplectic manifold
whose primitive defines an outward transverse Liouville vector field.
A contact form is hypertight if it has no contractible closed Reeb orbit.

\begin{proposition}[The Liouville exterior and its boundary]\label{prop:liouville}
For all sufficiently large $k>0$, the form
$\Omega_k=\omega_1+k\omega_2$ makes the rounded exterior $X$ a
Liouville domain. Its boundary contact structure admits a contact form
whose every closed Reeb orbit, including every iterate, represents an
infinite-order element of $\pi_1(\partial X)$.
\end{proposition}
\begin{proof}
The normal bundle of $\Sigma_O$ has flat rotational holonomies
$-I$ and $I$. On its punctured disk bundle let $\eta$ be the closed
angular connection with normal-circle period one, and $I$ the enclosed
fiber area. With $p_\Sigma$ denoting projection to the section,
\[
 \omega_1=dI\wedge\eta,\qquad
 \omega_2=p_\Sigma^*\sigma_\Sigma,
\]
where $\sigma_\Sigma$ is its positive area form. Remove the bundle
$I<\epsilon$, with $0<\epsilon<1$. The outward direction at this
new face decreases $I$.

Let $\tau:Y\to\mathbb R/\Z$ be the original boundary mapping-torus
coordinate and $b=\int_{\Sigma_O}\omega_2>0$. Choose a primitive
$\lambda_2$ of $\omega_2$ on $V$. The pullback of $\omega_2$
to $Y$ is zero, since the base area form restricts to zero on the
one-dimensional base boundary. Hence $\lambda_2|_Y$ is closed
and has period $b$ on the section boundary. Since $H_1(Y)=\Z$,
exact modifications give, simultaneously,
\[
 \lambda_2|_Y=b\,d\tau,\qquad
 \lambda_2=p_\Sigma^*\beta_\Sigma\text{ near }\Sigma_O,
 \quad d\beta_\Sigma=\sigma_\Sigma,
 \quad\beta_\Sigma|_{\partial\Sigma_O}=b\,d\tau.
\]
For completeness, first take $\beta_\Sigma=\lambda_2|_{\Sigma_O}$
and adjust its boundary value by an exact form. The difference
$\lambda_2-p_\Sigma^*\beta_\Sigma$ is closed and zero on the section,
hence exact on its neighborhood. Remove it there. The remaining exact
boundary difference has a primitive zero on that neighborhood and can
be removed using a collar supported away from the section.

A primitive $\lambda_1$ of $\omega_1|_X$ has normal-circle period
$\epsilon-1$, by Stokes on a punctured fiber of area $1-\epsilon$
and oriented boundary $-\mu$. A further exact modification thus gives
\[
 \lambda_1=(I-1)\eta+p_\Sigma^*\zeta,
 \qquad d\zeta=0,
\]
on a collar of the normal face. Set $\lambda_k=\lambda_1+k\lambda_2$.
On that face,
\[
 \lambda_k\wedge d\lambda_k
   =k(\epsilon-1)\eta\wedge p_\Sigma^*\sigma_\Sigma>0
\]
with its outward orientation. On the original boundary face,
\[
 \lambda_k\wedge d\lambda_k
 =\lambda_1\wedge\omega_1+kb\,d\tau\wedge\omega_1>0
\]
for sufficiently large $k$, by compactness. The Liouville vector field
is outward on both faces. A monotone rounding preserves the two strict
outward inequalities.

We can choose the rounding and primitive more precisely. Near the
section boundary take a base area coordinate $J_\Sigma$ with
$J_\Sigma=b$ at its boundary and
$\beta_\Sigma=J_\Sigma d\tau$. The flat boundary framing gives
$\eta=d\vartheta$ there. Since $\int_{\partial\Sigma_O}\zeta=0$,
remove its exact restriction on this collar. We obtain
\[
 \lambda_k=(I-1)d\vartheta+kJ_\Sigma d\tau.
\]
These exact modifications preserve contactness for one sufficiently
large $k$. Round along a curve $(I(q),J_\Sigma(q))$ with both
derivatives nonnegative and not simultaneously zero, in the region
$I<1$, $J_\Sigma>0$. On the resulting torus collar put
$f=I-1<0$, $g=kJ_\Sigma>0$. Its Reeb field is
\begin{equation}\label{eq:corner-reeb}
 \Delta=fg'-f'g<0,\qquad
 R_{\lambda_k}=\frac{g'\partial_\vartheta-f'\partial_\tau}{\Delta}.
\end{equation}
It has no $q$ component and is a nonzero linear field on each torus.

Topologically $\partial X$ is the union of the mapping torus
$\mathcal A$ of the punctured genus-two fiber and
$\mathcal B_\Sigma=\Sigma_O\times S^1$, meeting in a torus $T$.
Both maps from $\pi_1(T)$ inject: on $\mathcal A$, use mapping-torus
height and the nontrivial boundary word in the free fiber group; on
$\mathcal B_\Sigma$, use the nontrivial boundary word of the punctured
torus and its independent circle factor. Amalgam normal forms embed
both pieces and $T$ into $\pi_1(\partial X)$.

Choose invariant level tori in the normalized collar. A closed orbit
cannot cross them. In $\mathcal A$ it has positive mapping-torus
degree, since the Reeb direction is the positively oriented suspension
direction. In $\mathcal B_\Sigma$ it is a nonzero negative multiple
of the normal circle, since
$R_{\lambda_k}=(\epsilon-1)^{-1}\partial_\vartheta$ there.
In the collar it has nonzero integral displacement in $\pi_1(T)$
by \eqref{eq:corner-reeb}. Each class has infinite order in its piece
and hence in the boundary group. The argument includes all iterates.
\end{proof}

This boundary is universally tight. Indeed hypertightness lifts to
finite covers. Hofer's theorem \cite{Hof93} would produce a
contractible closed Reeb orbit for an overtwisted contact structure
on any such closed cover, so all those covers are tight. Residual finiteness of
three-manifold groups then places any compact overtwisted disk from
the universal cover into a finite cover, a contradiction; see
\cite[Lemma 4.23]{CP25} for this covering argument.
In particular no $\pi_1$-injective torus in this boundary is a
\emph{mixed torus} in the sense of \cite{CM25}: such a torus has a
virtually overtwisted toric neighborhood, whose universal cover would
embed in the universal cover of the boundary. The natural interface
$T$ is incompressible and therefore cannot supply the mixed torus
required by that splitting theorem.

The two pieces also give the natural spinal open book of this contact
boundary, with genus-two pages and punctured-torus spine
\cite{LVW19}. Thus this particular open book does not satisfy the
planar-page hypotheses of the corresponding filling classification
\cite{LVW20}. Neither observation rules out every possible contact
decomposition. No Weinstein structure on $X$ is asserted here.
Most importantly, the injections just proved are into
$\pi_1(\partial X)$, not into $Q_X$ or $Q_X/C_\mu$.
Hypertightness is not a substitute for either missing injection.

\subsection{Why exactness does not give a transverse primitive}

One plausible use of the symplectic pair would require a primitive of
$\omega_2$ annihilating its kernel foliation. This stronger exactness
condition fails for the particular pair constructed above.

\begin{proposition}\label{prop:basic}
There is no one-form $\xi$ on $V$ with $d\xi=\omega_2$ and
$\xi|_{\ker\omega_2}=0$.
\end{proposition}
\begin{proof}
On the exterior of the two surgery neighborhoods, Cartan's formula
would make $\xi$ invariant along the fibers as well as zero on their
tangent spaces. Those fibers are connected. Indeed the two disjoint
curves removed from the fiber have independent homology classes
$-[x]-[r]$ and $-[r]$; a disconnected complement would supply a
nonzero homology relation by taking the boundary of a complementary
component. Thus $\xi$ would descend to a primitive of the base area
form, using the section to identify its transverse values.

In the beta insertion put $q=Q_2\in\mathbb R/\Z$, $p=Y_2$.
A primitive annihilating $\ker\omega_2$ is independent of $Q_1,Y_1$.
Its values on the upper and lower gluing sides must agree, so the
base primitive must satisfy
\[
 \mathfrak h^*\xi=\xi,\qquad d\xi=dq\wedge dp,
 \qquad \mathfrak h(q,p)=(q+\chi(p),p).
\]
Write $\xi=a(q,p)dq+b(q,p)dp$ and let
$F(p)=\int_0^1a(q,p)\,dq$. Integrating the $dp$ coefficient of
the invariance equation gives $\chi'(p)F(p)=0$, whereas integrating
differential exactness gives $F'(p)=-1$. The nonconstant smooth step
$\chi$ has derivative nonzero on an interval. The two equations
contradict each other there.
\end{proof}

This does not contradict ordinary exactness of $\omega_2$, or the
local primitive normalizations used for the Liouville boundary. It
rules out a global transverse Liouville field obtained from this pair
by the stronger annihilation condition.

\subsection{Two handles realizing the central quotient}\label{subsec:handles}

In fact the central extension has a four-dimensional exact symplectic
realization. To distinguish it from the double $Z$, denote this
auxiliary domain by $\mathcal Z$.

\begin{proposition}[A four-dimensional central-extension realization]\label{prop:two-handles}
The map $\pi_1(\partial X)\to Q_X$ is surjective, and
\[
 C_\mu=\nc{[\mu,x],[\mu,y]}_{Q_X}.
\]
Two Weinstein two-handles can be attached to $X$ to give a Liouville
domain $\mathcal Z$ with
\[
 \pi_1(\mathcal Z)=\widehat G,\qquad
 H_1(\mathcal Z)=H_3(\mathcal Z)=0,\qquad H_2(\mathcal Z)=\Z^4.
\]
A retained normal-circle Reeb orbit $\gamma$ represents $t^{-1}$.
The handles may be chosen so that $c_1(T\mathcal Z)=0$.
\end{proposition}
\begin{proof}
The marked $x,y$ can be transported to the original boundary along a
base arc in $v,w<1/2$, missing both surgery tracks. Fix their
common basepoint on the original boundary. For the normal-face loop
below, compare its radial basing path with a path from this boundary
basepoint: radial transport times the same base arc is a square
contained in $v,w<1/2$, disjoint from both surgery tracks. Its
boundary homotopy identifies the two basing routes in $X$.
Thus all three loops below are images of loops with this one boundary
basepoint, rather than independently conjugated boundary images.
On the normal face,
follow the $A$ base loop using flat normal transport and close it by
the positive inner semicircle from the opposite radial point back to
the initial one. Based by the radial arc, this loop is $A[r,s]$:
the comparison path retracts to $E_5E_6E_7E_8$. The transport
rectangle misses the alpha track, and the beta chord $e$ misses
both radial arcs. Thus no additional meridian occurs.
The exterior relations give
\[
 [r,N^{-1}s]=[r,s],\qquad A[r,s]A^{-1}=[x,y]=J,
\]
using $[N,E]=1$ and the $Ar,As$ rows. Hence this boundary loop is
$U_\partial=JA$. Together with $x,y$ it generates $Q_X$, since
$A=J^{-1}U_\partial$.

For the centralization assertion use the free identity
\[
 [JA,\mathcal R]=\mathcal R[A^2,J]\mathcal R^{-1}.
\]
Thus $JA$ commutes with $\mu=\mathcal R$ already in $Q_X$.
Imposing $[\mu,x]=[\mu,y]=1$ makes $\mu$ commute with $J$,
and consequently with $A=J^{-1}(JA)$. Since $A,x,y$ generate the
whole group, this centralizes $\mu$ against every element. The reverse
normal-subgroup inclusion is immediate.

Choose disjoint embedded knot representatives of these two boundary
commutators, away from one normal-circle Reeb orbit. Legendrian
approximation and the exact two-handle construction \cite{Wei91}
give two Weinstein handles. Their attachment kills precisely the
stated two conjugacy classes. The contact form near the chosen orbit
can be retained, so its class becomes $t^{-1}$. The relative handle
chain complex has $\Z^2$ only in degree two. Together with
Proposition~\ref{prop:exterior}, this gives the stated homology.

Since $H_2(X)\to H_2(V)$ is zero and $H_1(X)=0$, restriction gives
$c_1(TX)=0$ integrally. Fix a determinant trivialization and use
the induced contact-plane trivialization to measure rotation. Positive
and negative stabilizations change a Legendrian rotation number by
$1$ and $-1$, so the two knots can be given rotation zero. Their free
homotopy classes and the marked orbit are unchanged. The relative
first Chern obstruction of each handle is that rotation number,
by the local obstruction computation in the proof of
\cite[Proposition 2.3]{Gom98}; hence the determinant trivialization
extends over both handles. These choices change auxiliary handle
framings, which were not prescribed. They change none of the original
two surgery slopes defining $V$.
\end{proof}

This construction uses a Weinstein attaching cobordism, and does not
assert that $X$ or $\mathcal Z$ itself is Weinstein. Since
$\widehat G$ is superperfect, the Hopf sequence makes
$\pi_2(\mathcal Z)\to H_2(\mathcal Z)=\Z^4$ surjective.
Thus $\mathcal Z$ is not aspherical. Exactness does not exclude
these mapped spherical classes.

\subsection{The orbit index and its limitation}
The retained orbit has identity linearized return map, so its relevant
index is the Robbin--Salamon index, not a nondegenerate
Conley--Zehnder index. Normalize a positive full rotation of a
symplectic plane to have index two. With the determinant trivialization
of Proposition~\ref{prop:two-handles},
\begin{equation}\label{eq:orbit-index}
                         \mu_{\mathrm{RS}}(\gamma^n)=-6n
                         \qquad(n\ge1).
\end{equation}
Here is the framing calculation. A fiber punctured at the section has
genus two, Euler characteristic $-3$, and oriented boundary
$\gamma=-\mu$. Its normal bundle has the product frame $n_0$ from
the base disk. An exact change near the chosen base point arranges that
the tangential base term of the primitive vanishes there; the Liouville
vector $v$ along $\gamma$ is then the outward tangent normal of the
punctured fiber. If $\mathfrak d$ is the determinant trivialization, write
$v\wedge n_0=a\mathfrak d$ on $\gamma$. The relative Euler number gives
$\deg a=-3$. The contact frame $e_0$ induced by $\mathfrak d$ satisfies
$n_0=ae_0$. Pure circle translation preserves $n_0$, so in the
$e_0$ frame the return path has winding $-3$. This proves
\eqref{eq:orbit-index} and its iterate formula. Since
$H^1(\mathcal Z;\Z)=0$, the determinant trivialization is unique
up to homotopy. If an iterate contracts in $\mathcal Z$, its
capping-disk index is the same, independently of the disk, because
$c_1(T\mathcal Z)=0$.

The index does not preclude such a disk. A useful comparison is the
Brieskorn link
\[
 \{z_0^6+z_1^6+z_2^6=0\}\cap S^5.
\]
Its periodic contact form has simple Hopf-circle orbits of index $-6$
and iterate indices $-6n$, computed in the Milnor-filling trivialization
\cite[Sections 2--3 and 5.3, Proposition 5.9]{KvK16}. The Milnor
filling is homotopy equivalent to a wedge of $125$ two-spheres and has
$c_1=0$, so every orbit contracts there. The precise periodic boundary
form can be obtained from a boundary form $f\alpha$ by taking the
graph $s=\log(C/f)$ in the positive symplectization, with $C>\max f$.
Yet the link is a circle bundle over the smooth plane sextic of genus
ten; its circle subgroup injects because that base is aspherical.
The periodic form is therefore hypertight. This comparison rules out
an implication from exact fillability, $c_1=0$, hypertightness and the
indices \eqref{eq:orbit-index} alone. It does not have the topology or
the complete presentation of our domain.

\subsection{What a mixed-torus decomposition would still require}
There is also a general homological restriction on applying the
splitting theorem of \cite[Theorem 1.1]{CM25}. If a connected exact
filling $L$ has $H_1(L;\Z)=0$, splitting it along a mixed torus
produces exactly two filling components, and the torus separates its
boundary. Indeed the reconstruction attaches one round one-handle
\[
 (S^1\times D^1\times D^2,\ S^1\times S^0\times D^2).
\]
Its relative $H_1$ is $\Z$. In the pair sequence the boundary of its
generator is the difference of the two attaching component classes.
Since $H_1(L)=0$, these components must be distinct. A single round
handle can connect only two components. Neither can be closed, by
Stokes' theorem for an exact symplectic form. A nonseparating torus
would instead give a connected cut-and-capped boundary, a contradiction.

This applies both to $X$ and to the auxiliary $\mathcal Z$.
For $X$ it complements the universal-tightness obstruction: any mixed
torus would have to be separating and fail to be $\pi_1$-injective.
If each auxiliary attaching knot is stabilized once more in both signs,
its rotation remains zero and the surgery supplies the mixed tori of
\cite[Theorem 1.4]{CM25}. Applying that theorem twice reduces an exact
filling of the resulting boundary to some exact filling of $\partial X$.
It does not identify that filling with the marked $X$ or prove uniqueness
of its fundamental group. For the filling constructed here, reattaching
the two handles restores exactly $Q_X/\nc{[\mu,x],[\mu,y]}$.
The same centralization problem remains.

\section{The remaining questions and their logical relations}\label{sec:open}

The established constructions can be organized as follows:
\[
\begin{gathered}
 \Pi\hookrightarrow Q_c\hookrightarrow Q_2\hookrightarrow K
       \twoheadrightarrow H_\beta\twoheadrightarrow G,\\
 Q_X\twoheadrightarrow\widehat G\twoheadrightarrow G,
 \qquad \pi_1(\mathcal P)=\pi_1(\mathcal Z)=\widehat G.
\end{gathered}
\]
The first three inclusions are proved before the final filling.
The kernel of $K\to H_\beta$ is free with the module calculated
in Proposition~\ref{prop:transport}; the remaining kernel is unknown.
The second line describes centralization of the section meridian,
followed by killing its central image. These two lines isolate
different possible ways to decide the same complete group.

The homomorphism $h:K\to\Z$ taking $A$ to $1$ and $Q_2$ to
$0$ is well defined. Since $h(\lambda)=1$, the final attaching
word is nontrivial in $K$. The complete group is obtained only
after it is imposed.

\begin{question}\label{q:decisive}
In the group $K$ of Theorem~\ref{thm:last}, is
\[
                              J\in\nc{\lambda}_K?
\]
Equivalently, do there exist finitely many $h_i\in K$ and
$\varepsilon_i\in\{1,-1\}$ with
\begin{equation}\label{eq:missing-disk}
                J=\prod_{i=1}^n h_i\lambda^{\varepsilon_i}h_i^{-1}?
\end{equation}
\end{question}

An equality \eqref{eq:missing-disk} would prove $G=1$ by
Proposition~\ref{prop:J} and complete the conditional conclusions
of Corollary~\ref{cor:conditional}. A proof of nonmembership would
show $G\ne1$. By Proposition~\ref{prop:extension}, it would also
show that this particular double is not simply connected and that
this particular $W$ is not homeomorphic to $\BK$.

The stronger statement
\begin{equation}\label{eq:strong-injection}
                         \nc{\lambda}_K\cap\Pi=\{1\}
\end{equation}
would prove nonmembership, but it is not established. The surface
injection \eqref{eq:surface-injection} precedes the final quotient
and supplies no such conclusion by itself. In fact, the whole
vertex group $Q_2$ cannot inject under the actual map $K\to G$.
The element
\[
                         \kappa_*=a_*xa_*^{-1}r^{-1}
\]
dies there, because $\lambda=1$ sets $A=a_*$ and the $A$ edge
relation gives $AxA^{-1}=r$. On the other hand,
$H_1(Q_2)=\Z\langle r,s,B\rangle$: its relations kill $x,y$
and impose no relation on these three classes. The class of
$\kappa_*$ is $-r$, so it is nontrivial in $Q_2$. A theorem
asserting injection of that entire HNN base cannot apply to this
quotient. This obstruction does not settle the smaller surface
injection question \eqref{eq:strong-injection}.

The intermediate free-kernel result likewise does not control the
new relations after $\lambda$ is imposed. One must either produce
an equality \eqref{eq:missing-disk} in the complete intermediate
group or give an argument excluding every such equality. An explicit
representation of all fifteen relations with nonidentity $J$ would
also prove nonmembership. Failure to find such a representation in
any chosen collection does not prove membership.

\subsection*{Sufficient statements and unresolved hypotheses}
For the central route, the following three statements are equivalent:
\[
 t\ne1\text{ in }\widehat G
 \quad\Longleftrightarrow\quad \mu\notin C_\mu
 \quad\Longleftrightarrow\quad
 \operatorname{im}\bigl(\pi_2(V)\xrightarrow{\langle e,-\rangle}\Z\bigr)
                         \ne\Z.
\]
Each implies $G\ne1$. Equivalently the retained orbit $\gamma$ does
not bound a mapped disk in $\mathcal Z$. None follows from its
infinite order in the boundary, the index calculation, or exactness
of the filling. Conversely $t=1$ would leave the decision of $G$
open. A proof about arbitrary spherical classes or arbitrary relations
in \eqref{eq:universal-central} is required.

For a triviality argument, any one of the collapse relations in
Proposition~\ref{prop:collapse}, proved from the existing relations
rather than added as a hypothesis, would suffice. The cyclic-centralizer
assertion \eqref{eq:cyclic-centralizer} is another sufficient statement,
but its centralizer is taken after imposing the surface relation.
The known cyclic centralizer in the earlier trefoil amalgam does not
supply it. Likewise a new geometric decomposition would have to prove
its gluing-face injections, avoiding the two specific obstructions in
Section~\ref{subsec:cuts}.

These formulations are useful because they name the missing assertions
and retain the complete construction. They do not provide a theorem
whose verified hypotheses already settle one of them. The four unit
fillings remain geometrically eligible, but the known reflection changes
a section index and supplies no equivalence between the two opposite
sign choices with indices $(0,0)$. No alternative has acquired a proved
advantage for the final group calculation.

The present results establish marked geometric and algebraic structure,
and exclude several specific shortcuts. They do not indicate which
answer the final normal-closure question has. In particular neither
the failure of those methods nor unsuccessful representation searches
is evidence that $G$ is trivial or nontrivial. The exotic-manifold
conclusions remain conditional, and the fixed construction remains an
open problem.

\appendix
\section{Free identities for the presentation change}\label{app:identities}

We give the products used in Proposition~\ref{prop:five}, so that
the equivalence with every geometric relation can be checked without
a group-search algorithm. All identities in this appendix are free
word identities after the definitions \eqref{eq:dictionary} are
substituted. Put $g=\ell^{-1}x^{-1}$. Then
\begin{equation}\label{eq:MC-product}
                  [M,(rx)^{-1}]=g\mathcal T g^{-1}.
\end{equation}
For the beta peripheral commutator put
\[
 b_*=sz^{-1}s^{-1},\qquad
 u_*=b_*r^{-1}A,\qquad v_*=sz^{-1}r^{-1}A.
\]
One has
\begin{equation}\label{eq:NE-product}
\begin{split}
 [N,E]={}&(u_*\mathcal T^{-1}u_*^{-1})
           (u_*\mathcal R^{-1}u_*^{-1})\\
 &\quad\cdot((u_*JA)\mathcal R(u_*JA)^{-1})
           (v_*\mathcal T v_*^{-1}).
\end{split}
\end{equation}
If $w_*=s^{-1}E^{-1}$, then
\begin{equation}\label{eq:Br-product}
                      [B,r]=w_*[N,E]^{-1}w_*^{-1}.
\end{equation}
These products account for both the $[N,E]$ and $Br$ rows.
The alpha meridian word $AMA^{-1}K_0^{-1}$ is freely trivial.

The exact identity $K_0=(xA)^{-1}\mathcal T^{-1}$ gives
\eqref{eq:S0}. Finally, after the defining elimination $B=s^{-1}Ns$,
put $D=BK_0^{-1}$. Then
\begin{equation}\label{eq:beta-meridian-product}
\begin{aligned}
 \mathcal S_0&=DN^{-1},&F_\beta&=\mathcal S_0^{-1},\\
 B(K_0^{-1}NK_0)B^{-1}N^{-1}
      &=\mathcal S_0(N\mathcal S_0^{-1}N^{-1}).
\end{aligned}
\end{equation}
This accounts for the beta meridian transport as well as the
original beta filling. The rows used to define $r,s,M,N$ vanish
freely; the replacement defining $B$ was reversible by
\eqref{eq:beta-elimination}. The surface, $Ar$, $Bx$, and $By$
rows are exactly the first four retained relators in
\eqref{eq:five}; the fifth is recovered from
\eqref{eq:S0}. Hence the argument works in both directions.

The two free equalities in \eqref{eq:two-rows} provide the further
reversible replacement used to obtain $K/\nc{\lambda}_K$.

\subsection*{Verification of the displayed algebra}
The word identities can be verified by substituting the printed
definitions and cancelling adjacent inverse letters. An accompanying
Python program\footnote{The accompanying \texttt{verify.py} uses only
the Python standard library; run \texttt{python3 -B verify.py} from
the directory containing the ancillary files.} performs precisely these checks, reconstructs the
fifteen original rows, and checks the integral exponent calculations.
It also checks the two cotangent normal matrices and the two
rank-three homology triples. These finite verifications concern
identities and integer linear algebra. The geometric inclusions,
HNN injectivity, and classification arguments are the proofs and
cited theorems in the text. No finite-quotient computation is used
to decide $G$. The companion \texttt{verify\_expanded.py}
checks the nineteen-factor identity, the lift normalization, the
alternative trefoil dictionary and labelled action, the Whitehead cut
table, and the displayed exterior-algebra identities. It imports only
arithmetic routines from the accompanying checker and uses no historical
research files.

\section{The identity used in the central extension}\label{app:central-identity}

To make Theorem~\ref{thm:central} checkable without a separate data
archive, we print its entire free identity. In this appendix only,
encode the generator $A$ by the lower-case letter \texttt{a}, and
encode inverses by upper-case letters. Thus \texttt{A} means $A^{-1}$,
whereas \texttt{X},\texttt{Y} mean $x^{-1},y^{-1}$.
The five relators in this alphabet are
\[
\begin{array}{c|l}
\mathcal R&\texttt{xyXYaxyXYA}\\
\mathcal T&\texttt{aaxAAX}\\
\mathcal X&\texttt{AYaayAxaYAAyay}\\
\mathcal Y&\texttt{AYaayxyaYAAyaXY}\\
\mathcal S&\texttt{AYaayAxyaaYA}.
\end{array}
\]
For each row of the following table take $v_i r_i^{\varepsilon_i}v_i^{-1}$
in its displayed order. Their product freely reduces to the empty word.
\begin{center}
\begin{adjustbox}{max width=\textwidth}
\begin{tabular}{rrcl}
\toprule
$i$&$\varepsilon_i$&$r_i$&$v_i$\\\midrule
1&1&$\mathcal X$&$1$\\
2&1&$\mathcal Y$&\texttt{Y}\\
3&-1&$\mathcal X$&\texttt{xAYaayAXaYAAya}\\
4&-1&$\mathcal Y$&\texttt{xyAYaayAYXYAAya}\\
5&-1&$\mathcal T$&\texttt{xyXYaxAAX}\\
6&-1&$\mathcal R$&\texttt{xyXYaxAAXYa}\\
7&1&$\mathcal R$&\texttt{xyXYaxAAYX}\\
8&1&$\mathcal T$&\texttt{xyXYaxAAYX}\\
9&1&$\mathcal T$&\texttt{xyXYaxAAYaayA}\\
10&1&$\mathcal S$&\texttt{xyXYaxA}\\
11&-1&$\mathcal T$&\texttt{xyXYaxyAAYaayXYAAyaaxAAYX}\\
12&-1&$\mathcal R$&\texttt{xyXYaxyAAYaayXYAAyaaxAAYX}\\
13&1&$\mathcal R$&\texttt{xyXYaxyAAYaayXYAAyaaxAAXYa}\\
14&1&$\mathcal T$&\texttt{xyXYaxyAAYaayXYAAyaaxAAX}\\
15&-1&$\mathcal S$&\texttt{xyXYaxyXAAYaaXAYAAya}\\
16&-1&$\mathcal T$&\texttt{xyXYaxyXAAYaaXAA}\\
17&1&$\mathcal R$&$1$\\
18&1&$\mathcal T$&\texttt{AYaayA}\\
19&-1&$\mathcal R$&\texttt{AYaayAxa}\\\bottomrule
\end{tabular}
\end{adjustbox}
\end{center}
The signed counts are $(0,1,0,0,0)$ in the order
$\mathcal R,\mathcal T,\mathcal X,\mathcal Y,\mathcal S$.
This is the identity used to force the $\mathcal T$ value to vanish
in Theorem~\ref{thm:central}.

\section{A trefoil amalgam description}\label{app:trefoil}

A different exact presentation exposes a trefoil subgroup before two
remaining relations are imposed. We give the dictionary and the
subgroup calculation because neither the trefoil center nor the
infinitude of this intermediate group decides the complete quotient.
Throughout this appendix the filling is still $W(+1,-1;0,0)$.

\subsection{Two generators and the two remaining equations}
On free generators $A,s$ define
\begin{equation}\label{eq:two-generator-dictionary}
\begin{aligned}
 y&=A^{-1}sA,&z&=A^2,&N&=[s,z^{-1}],&B&=s^{-1}Ns,\\
 k&=[s,B^{-1}],&r&=A^{-1}k^{-1},&x&=A^{-1}rA,&M&=rkr^{-1},\\
 R_0&=[x,y][r,s],&T_0&=[A^2,x],&X_0&=BxB^{-1}y,\\
 Y_0&=ByB^{-1}(BMB^{-1}yx)^{-1}.&&&&
\end{aligned}
\end{equation}
Then
\begin{equation}\label{eq:two-generator}
              G=\langle A,s\mid R_0,T_0,X_0,Y_0\rangle.
\end{equation}
To verify reversibility, the $As$ row solves for $N$, and the beta
filling with the $Bs$ row solves for $B$ and $k$. The alpha filling
then gives $r=A^{-1}k^{-1}$, while the clean $Ax$ row gives $x$.
The still-required $Ar$ row is $T_0=1$. The peripheral and meridian
redundancies are precisely the free products of
Appendix~\ref{app:identities}; retain $R_0,T_0$ when using them.
The two remaining beta transports are $X_0,Y_0$. Reversing these
eliminations recovers every original row.

Put $A=u$, $s=ub$, and introduce
$a=u^{-1}bu$, $c=ubu^{-1}$, $d=u^2bu^{-2}$.
There is a free identity, with $e=u^{-3}bu^3$,
\[
                  BT_0B^{-1}X_0=ca^{-2}ea.
\]
Conjugating its right side by $u^2$ shows that the single equation
$BT_0B^{-1}X_0=1$ defines
\[
 \mathcal H=F(a,b,c,d)\rtimes_\Phi\langle u\rangle,
 \quad
 \Phi(a,b,c,d)=(b,c,d,c^{-1}a^{-1}c^2).
\]
Its inverse is $(a,b,c,d)\mapsto(b^2d^{-1}b^{-1},a,b,c)$.
Retaining $T_0=1$ as well gives $\mathcal P_-=\mathcal H/\nc{T_0}$.
In these coordinates $A=u$, $s=cu$, $y=bu$.
If $f=a^{-1}cac^{-1}a^{-1}$, the relation $T_0=1$ is equivalent,
in $\mathcal H$, to
\[
 \sigma_0=\Phi(f)b^{-1}f^{-1}b=1,
                   \quad\text{or equivalently }[ua,f]=1.
\]
For a direct word comparison, $\sigma_0$ is conjugate to $T_0$
by $b^{-1}a^{-1}cu$; substitution in the free-by-cyclic normal form
verifies this equality. Thus the transport group
\[
 \mathcal P_-:=\langle A,s\mid T_0,X_0\rangle
                   =\mathcal H/\nc{\sigma_0}
\]
retains exactly those two equations.

The remaining words are fully specified by
\[
\begin{aligned}
 k_f&=ac^{-1}bdb^{-1},&
 J_0&=k_f^{-1}a\Phi(k_f)b^{-1},\\
 R_f&=J_0\Phi(J_0),&
 \eta&=a^{-1}cac^{-1}\,bdb^{-1}\,c^{-1}ac\,
           bd^{-1}b^{-1}\,ca^{-2}c\,d^{-1}b^{-1}.
\end{aligned}
\]
Indeed $x_0=k_f^{-1}u^{-1}=T_0x$ in $\mathcal H$, so $J_0$
is the original $J$ after $T_0$ is imposed. The original $Y_0$
has normal form $\eta^{-1}u$. Consequently
\begin{equation}\label{eq:trefoil-final}
 G=\mathcal P_-/\nc{R_f,\eta^{-1}u}
  =\langle a,b,c,d\mid \sigma_0,R_f,
       \eta i\eta^{-1}=\Phi(i)\ (i=a,b,c,d)\rangle.
\end{equation}
Eliminating $u$ retains all four action equations; it does not turn
this into a quotient with only $\sigma_0,R_f$.

\subsection{An amalgamated product with injective edge maps}
Let
\[
 \mathcal K=\langle a,c,z\mid zaz^{-1}=c,
                              zcz^{-1}=c^{-1}a^{-1}c^2\rangle,
 \qquad g=afa^{-1}.
\]
This is a free-by-cyclic group on $F(a,c)$, with $z=u^2$ in
$\mathcal H$. Put $v=cz$ and $h=av$. Reversible substitution gives
\[
 \mathcal K=\langle h,v\mid hvh=vhv\rangle,
 \quad a=hv^{-1},\quad c=vhv^{-2},\quad z=v^2h^{-1},
\]
the trefoil group. The following explicit action proves that
$\mathcal D_0=\langle z,f,g\rangle$ is free on those generators.

In the modular group $\langle\alpha,\beta\mid\alpha^2=\beta^3=1\rangle$
set $h=\beta^{-1}\alpha$, $v=\alpha^{-1}\beta^2$. These satisfy
the braid relation and therefore give a homomorphism from $\mathcal K$.
Let $L=F(e_1,e_2,e_3)$. On $L\times\{0,\ldots,29\}$ define
right actions by
\[
 (w,i)\alpha=(w\ell_\alpha(i),p(i)),\qquad
 (w,i)\beta=(w\ell_\beta(i),q(i)).
\]
Here are the complete permutation arrays, split at state $15$:
\begin{center}
\begin{adjustbox}{max width=\textwidth}
\begin{tabular}{cl}
\toprule
$p(0),\ldots,p(14)$&$1,0,17,4,3,8,7,6,5,10,9,27,13,12,28$\\
$p(15),\ldots,p(29)$&$16,15,2,24,20,19,29,23,22,18,26,25,11,14,21$\\
$q(0),\ldots,q(14)$&$25,2,3,1,5,6,4,0,9,16,12,10,11,15,13$\\
$q(15),\ldots,q(29)$&$14,8,19,17,18,22,20,21,24,26,7,23,28,29,27$\\
\bottomrule
\end{tabular}
\end{adjustbox}
\end{center}
All labels are one except
\[
 \ell_\alpha(12)=e_1^{-1},\quad\ell_\alpha(13)=e_1,
 \quad\ell_\alpha(22)=e_2,\quad\ell_\alpha(23)=e_2^{-1},
 \quad\ell_\beta(4)=e_3,\quad\ell_\beta(6)=e_3^{-1}.
\]
The permutations and orbit labels satisfy $\alpha^2=\beta^3=1$.
The words $z,f,g$ return state zero to itself with labels
$e_3,e_3e_1^{-1},e_2e_3$, respectively. These form a free basis of
$L$, with inverse expressions $e_1=f^{-1}z$, $e_2=gz^{-1}$,
$e_3=z$. Any relation among $z,f,g$ would therefore be a relation
among a free basis. This proves the claimed subgroup injection;
faithfulness of the whole modular action is unnecessary.

Now the map
\[
 \mathcal D_0\longrightarrow F(u,F_0),\qquad
       z\mapsto u^2,\quad f\mapsto F_0,\quad g\mapsto u^{-1}F_0u
\]
is injective: its images form a Schreier basis of the subgroup of
even $u$ parity. The resulting amalgam is
\begin{equation}\label{eq:finite-amalgam}
                    \mathcal P_-\cong\mathcal K*_{\mathcal D_0}F(u,F_0).
\end{equation}
Indeed identifying $F_0=f$ leaves precisely $u^2=z$ and
$u^{-1}fu=afa^{-1}$, the latter being $[ua,f]=1$.
Both directions are presentation eliminations, and both edge maps
are injective by the preceding free-basis calculation. Thus $\mathcal K$
embeds in $\mathcal P_-$.

This also supplies a complete membership test in the edge group.
Apply a word of $\mathcal K$ to state zero. If it moves that state,
it is outside $\mathcal D_0$. If it returns, convert its label to the
unique candidate word on $z,f,g$ by the inverse expressions above,
and compare that candidate with the original word using the
$F(a,c)\rtimes\langle z\rangle$ normal form. Equality is necessary
and sufficient. In the other factor membership is even $u$ parity.
Thus the amalgam normal form is effective without a bounded search.

The height kernel of $\mathcal P_-\to\Z$ is generated by $a,b,c,d$
and is nontrivial, perfect, and torsion-free. For perfectness, four
translates of its defining relation give the exponent vectors
\[
 (-1,1,0,0),\quad(0,0,-1,1),\quad(0,-1,1,0),\quad(-1,0,1,-1),
\]
which span $\Z^4$. Nontriviality follows from the embedded $F(a,c)$;
torsion-freeness follows from the injective amalgam of torsion-free
groups. In particular this intermediate group is not locally indicable.
These conclusions precede the two relations in \eqref{eq:trefoil-final}.

\subsection{Failure of height-kernel injection and simplicity}
The whole height kernel cannot survive the remaining height-one
relation, even though its free factors embed before that relation.
To see this precisely, let $E_0=F(a,c)$, $O_0=F(b,d)$, let
$j:E_0\to O_0$ send $(a,c)$ to $(b,d)$, and set
$S=\Phi^2|_{E_0}$. The subgroup
\[
 U_0=\mathcal D_0\cap E_0
    =F(f_n,g_n:n\in\Z),\qquad f_n=S^n(f),\quad g_n=S^n(g),
\]
has the displayed free basis by Schreier rewriting of
$F(z,f,g)\to\Z$. Define its automorphism
$\vartheta(f_n)=g_n$, $\vartheta(g_n)=f_{n-1}$.
The kernel $N_0=\ker(\mathcal P_-\to\Z)$ is exactly
\begin{equation}\label{eq:height-amalgam}
                      N_0=E_0*_{U_0}O_0,
\end{equation}
with edge maps inclusion and $j\vartheta$. Indeed
$f\,j(g)^{-1}=b\sigma_0^{-1}b^{-1}$ and its $\Phi$-image is
$j(f)S(g)^{-1}$. Their even and odd translates impose exactly
$f_n=j(g_n)$ and $g_{n+1}=j(f_n)$ for all $n$. This proves the
presentation and both edge injections for the entire infinite family.
Both edge images are proper, since they have infinite free rank.

The factors $E_0,O_0$ are not conjugate in $N_0$: they fix vertices
of different types in its Bass--Serre tree, and properness of the edge
subgroups makes each fixed vertex unique. If $N_0$ injected into
$\mathcal P_-/\nc{w^{-1}u}$ for any $w\in N_0$, the equation
$u=w$ would force $\Phi$ to be conjugation by $w$ on $N_0$.
But $\Phi(E_0)=O_0$, contradicting that nonconjugacy. Thus no such
height-one filling preserves the whole kernel. In the actual filling,
\[
                      \kappa_a=\eta a\eta^{-1}b^{-1}
\]
is an explicit killed element of infinite order. Its fourteen
alternating syllables in \eqref{eq:height-amalgam} are all outside the
edge groups: the six syllables of $\eta$, after relabeling the odd
factor by $j^{-1}$, end at states $3,16,18,21,15,2$ in the printed
modular action, and $a$ ends at state $3$. Thus the word is cyclically
reduced. It dies because $u=\eta$ and $uau^{-1}=b$.

Nor can failure of injection be promoted to collapse by assuming that
$N_0$ is simple. In fact $N_0$ is not simple. Here is the specific
intersection calculation supporting that conclusion. Put
$v_0=c^{-1}a^2c^{-1}\in E_0$. Project $\mathcal K$ to the modular
group as above. The based product of the folded subgroup graphs for
$\mathcal D_0$ and $v_0\mathcal D_0v_0^{-1}$ has ten vertices,
which may be labeled $q_0,\ldots,q_9$, and exactly the following edges:
\[
\begin{array}{c|l}
\alpha\text{ pairs}&(q_0,q_1),\ (q_2,q_4),\ (q_5,q_7)\\
\beta\text{ cycles}&(q_1,q_3,q_2),\ (q_4,q_6,q_5),\ (q_7,q_9,q_8).
\end{array}
\]
Here is a direct way to reproduce the graph. First reduce the modular
words for $z,f,g$ and their $v_0$ conjugates. For each of the two
triples, attach the three based loops, completing every alpha edge
to a two-cycle and every beta edge to a three-cycle. Whenever two
edges with the same label leave a vertex, identify their endpoints;
repeat until each label defines a partial permutation. Completing a
finite-factor cycle uses only $\alpha^2=\beta^3=1$, and folding
preserves the subgroup represented by closed paths. Finally form the
product of the two resulting labelled graphs and keep the component
of the pair of basepoints. It is the ten-vertex graph above, up to
relabelling.\footnote{The functions \texttt{modular} and \texttt{fold}
in the accompanying \texttt{verify\_expanded.py} perform these steps
from the printed words and check the rooted labelled graph isomorphism.}
For an explicit identification, number the vertices of the first folded
graph by $0,\ldots,26$ and of the second by $0,\ldots,36$.
Their complete partial permutations are the following cycles; an omitted
vertex has no edge of that label. Both basepoints are $0$.
\begin{center}
\begin{adjustbox}{max width=\textwidth}
\begin{tabular}{cl}
\toprule
First graph, $\alpha$&$(0\ 1),(2\ 4),(5\ 7),(6\ 9),(10\ 12),(13\ 15),$\\
&$(16\ 11),(3\ 18),(19\ 21),(22\ 24),(25\ 20)$\\
First graph, $\beta$&$(2\ 1\ 3),(5\ 4\ 6),(7\ 0\ 8),(9\ 10\ 11),(12\ 13\ 14),$\\
&$(15\ 16\ 17),(18\ 19\ 20),(21\ 22\ 23),(24\ 25\ 26)$\\
Second graph, $\alpha$&$(0\ 1),(2\ 4),(5\ 7),(8\ 10),(11\ 13),(14\ 16),(17\ 19),$\\
&$(20\ 22),(23\ 25),(26\ 21),(24\ 28),(29\ 31),(32\ 34),(35\ 30),(18\ 12)$\\
Second graph, $\beta$&$(2\ 1\ 3),(5\ 4\ 6),(8\ 7\ 9),(11\ 10\ 12),(13\ 14\ 15),$\\
&$(16\ 17\ 18),(19\ 20\ 21),(23\ 22\ 24),(25\ 26\ 27),$\\
&$(28\ 29\ 30),(31\ 32\ 33),(34\ 35\ 36)$\\\bottomrule
\end{tabular}
\end{adjustbox}
\end{center}
The product-component dictionary is $q_i=(i,i)$ for $0\le i\le8$
and $q_9=(0,9)$, with basepoint $q_0$. Thus its ten vertices and
all their edges can be read directly from these two tables.

Every closed path in that component has trivial label: it consists
of three beta triangles connected by alpha edges, with one remaining
alpha edge to the basepoint. The projected intersection is therefore
trivial. The labelled action already proves injectivity of the
projection on $\mathcal D_0$, so
\[
 \mathcal D_0\cap v_0\mathcal D_0v_0^{-1}=1,
 \qquad U_0\cap v_0U_0v_0^{-1}=1.
\]
The amalgam criterion of \cite[Corollary 2.2]{MO15} now makes both
$\mathcal P_-$ and $N_0$ acylindrically hyperbolic: the edge inclusions
are proper and an edge group has finite intersection with one of its
conjugates. The virtually cyclic alternative is excluded by the
embedded $F(a,c)$. Here acylindrical hyperbolicity has its usual meaning
of admitting a non-elementary acylindrical action on a hyperbolic space;
it is not an assertion about this particular tree action.
Such groups have uncountably many normal subgroups
\cite[Theorem 3.7(f)]{MO15}, so $N_0$ is not simple. None of these
properties is asserted for the prescribed quotient $G$.

\subsection{What the alternative description does not imply}
Normal forms give, for every integer $j$,
\[
 C_{\mathcal P_-}(u^{2j+1})=\langle u\rangle.
\]
An odd power lies outside the normal index-two edge subgroup in the
free factor $F(u,F_0)$. Since that edge subgroup is normal, the
odd power lies in no conjugate of it inside this factor. It therefore
fixes exactly one vertex of the Bass--Serre tree. Its centralizer
preserves that vertex and lies in the free factor, where it is the
indicated cyclic group. This is not the centralizer needed after imposing the surface
relation.

To state that different question, put
$\mathcal Q=\langle A,s\mid R_0,T_0,X_0\rangle$ and
\[
 \varepsilon=(yx)^{-1}B(M^{-1}y)B^{-1}.
\]
Then $G=\mathcal Q/\nc\varepsilon$ and $[\varepsilon,y]=1$
in $\mathcal Q$. Indeed the corrected beta transports, applied to
\eqref{eq:crossing-identity}, give
$[y^{-1},yx\varepsilon]=[x,y]$; the free identity
\[
 [y^{-1},yx\varepsilon][x,y]^{-1}=x[\varepsilon,y]x^{-1}
\]
proves the commutation. The abelianization of $\mathcal Q$ is $\Z$,
with both $y$ and $\varepsilon$ mapping to its generator. Therefore
\begin{equation}\label{eq:cyclic-centralizer}
       C_{\mathcal Q}(y)\text{ cyclic}\quad\Longrightarrow\quad G=1.
\end{equation}
If that centralizer were generated by $v$, the primitive height of
$y$ would force $y=v^{\pm1}$, and the height of $\varepsilon$ would
force $\varepsilon=y$. Killing it then kills all generators by
\eqref{eq:two-generator-dictionary} and $X_0$. The cyclic-centralizer
hypothesis is unproved. Neither \eqref{eq:finite-amalgam} nor the
centralizer of odd powers of $u$ establishes it.

\section{A restriction on handlebody attaching curves}\label{app:whitehead}

The five conjugacy classes in \eqref{eq:five} cannot be represented
simultaneously by disjoint simple closed curves on the boundary of a
genus-three handlebody with its specified free fundamental group.
This statement allows arbitrary free basis changes. It concerns this
tuple of relators, not every possible presentation of $G$.

Apply the free automorphism $A\mapsto A$, $x\mapsto Ax$,
$y\mapsto yA^{-1}$ and cyclically reduce. In the alphabet of
Appendix~\ref{app:central-identity}, the resulting words are
\[
 \texttt{xyAXYaaxyAXY},\quad \texttt{aaxAAX},\quad
 \texttt{YaayAxaaYAAyyA},\quad \texttt{YaayxyaYAAyXY},\quad
 \texttt{YaayAxyaaYA}.
\]
Their total length is $56$. None is a proper power, as one checks
from its cyclic letters, and their distinct lengths distinguish their
maximal cyclic-subgroup classes. Define the Whitehead graph by an edge
from $w_i$ to $w_{i+1}^{-1}$ for every cyclic adjacency. For a signed
multiplier $v$, the change in length for its Whitehead cut $S$ is
$|\partial S|-\deg(v)$, where $v\in S$ and $v^{-1}\notin S$.
The following table prints every difference. Subsets of the other four
letters are ordered by increasing binary mask in the signed-letter
order $(\texttt{Y},\texttt{X},\texttt{A},\texttt{a},\texttt{x},\texttt{y})$,
omitting $v,v^{-1}$, with the first remaining letter corresponding
to the least significant bit.
\begin{center}
\begin{adjustbox}{max width=\textwidth}
\begin{tabular}{cl}
\toprule
$v$&Sixteen cut differences\\\midrule
\texttt{Y}&$0,8,15,17,15,19,10,8,2,10,9,11,15,19,2,0$\\
\texttt{X}&$0,17,21,26,23,28,24,17,11,24,24,25,18,19,11,0$\\
\texttt{A}&$0,7,4,9,2,1,6,3,11,14,7,8,11,6,7,0$\\
\texttt{a}&$0,7,6,11,8,7,14,11,3,6,1,2,9,4,7,0$\\
\texttt{x}&$0,11,19,18,25,24,24,11,17,24,28,23,26,21,17,0$\\
\texttt{y}&$0,2,19,15,11,9,10,2,8,10,19,15,17,15,8,0$\\\bottomrule
\end{tabular}
\end{adjustbox}
\end{center}
Whitehead minimality follows because all differences are nonnegative.
The graph contains all nine edges of $K_{3,3}$ with parts
$\{\texttt{Y},\texttt{X},\texttt{x}\}$ and
$\{\texttt{A},\texttt{a},\texttt{y}\}$. A length-minimal
Whitehead graph of a geometric multiword must be planar, by
Zieschang's theorem \cite{Zie65}; see
\cite[Section 2E1 and the discussion following Proposition 5.2]{Cas16}.
This proves the obstruction. Further changes of relators, as opposed
to free changes of generators, are outside its scope.

\end{document}